\documentclass{amsart}
\usepackage{amsmath,amssymb, amsbsy}
\usepackage[dvips]{graphicx}
\usepackage{color}
\usepackage[latin1]{inputenc}
\usepackage[active]{srcltx}
\usepackage{wrapfig}
\usepackage{graphicx}
\usepackage[usenames,dvipsnames]{pstricks}
\usepackage{epsfig}
 \usepackage{pst-grad} 
 \usepackage{pst-plot} 
\usepackage{epic}

\theoremstyle{plain}

\newtheorem{lemma}{Lemma}

\newtheorem{remark}{Remark}

\numberwithin{equation}{section}
\long\def\salta#1{\relax}

\newcommand{\dint}{\dyle\int}

\newcommand{\re}{{I\!\!R}}
\newcommand{\ren}{\re^N}

\newcommand{\dyle}{\displaystyle}

\newcommand{\x}{\times}
\newcommand{\irn}{\int_{\re^N}}
\newcommand{\io}{\int\limits_\O}

\newcommand{\limit}{\lim\limits}

\renewcommand{\a }{\alpha }
\renewcommand{\b }{\beta }
\renewcommand{\d }{\delta }
\newcommand{\D }{\Delta }
\newcommand{\e }{\varepsilon }

\newcommand{\g }{\gamma}

\renewcommand{\l }{\lambda }
\renewcommand{\L }{\Lambda }

\newcommand{\s }{\sigma }

\renewcommand{\O }{\Omega }

\newtheorem{Theorem}{Theorem}[section]

\newtheorem{Definition}[Theorem]{Definition}
\newtheorem{Lemma}[Theorem]{Lemma}
\newtheorem{Proposition}[Theorem]{Proposition}

\newcommand{\cqd}{{\unskip\nobreak\hfil\penalty50
        \hskip2em\hbox{}\nobreak\hfil\mbox{\rule{1ex}{1ex} \qquad}
        \parfillskip=0pt \finalhyphendemerits=0\par\medskip}}

\begin{document}
\title[On  Fractional quasilinear
parabolic problem]{On  Fractional quasilinear parabolic problem with Hardy potential }
\author[B. Abdellaoui, A. Attar, R. Bentifour \& I. Peral]{B. Abdellaoui$^*$, A. Attar$^*$, R. Bentifour$^*$ \& I.peral$^\dag$}

\address{\hbox{\parbox{5.7in}{\medskip\noindent {$*$Laboratoire d'Analyse Nonlin\'eaire et Math\'ematiques
Appliqu\'ees. \hfill \break\indent D\'epartement de
Math\'ematiques, Universit\'e Abou Bakr Belka\"{\i}d, Tlemcen,
\hfill\break\indent Tlemcen 13000, Algeria.\\[3pt]
$\dag$Departamento de Matem{\'a}ticas, U. Autonoma
de Madrid, \hfill\break\indent 28049 Madrid, Spain.\\[3pt]
        \em{E-mail addresses: }{\tt boumediene.abdellaoui@inv.uam.es, \tt ahm.attar@yahoo.fr, \tt rachidbentifour@gmail.com, \tt ireneo.peral@uam.es}.}}}}
\thanks{ This work is partially supported by projects
MTM2013-40846-P and MTM2016-80474-P, MINECO, Spain. } \keywords{Nonlinear nonlocal parabolic problems, Hardy potential, Caffarelli-Khon-Nirenberg inequality for degenerate weights, finite time extension, non existence result. \\
\indent 2010 {\it Mathematics Subject Classification:  35K59, 35K65, 35K67, 35K92, 35B09.} }


\maketitle

\begin{abstract}
The aim goal of this paper is to treat the following problem
\begin{equation*}
\left\{
\begin{array}{rcll}
u_t+(-\D^s_{p}) u &=&\dyle \l \dfrac{u^{p-1}}{|x|^{ps}}  &
\text{ in } \O_{T}=\Omega \times (0,T), \\
u&\ge & 0 &
\text{ in }\ren \times (0,T), \\
u &=& 0 & \text{ in }(\ren\setminus\O) \times (0,T), \\
u(x,0)&=& u_0(x)& \mbox{  in  }\O,
\end{array}%
\right.
\end{equation*}
where $\Omega$ is a bounded domain containing the origin,
$$ (-\D^s_{p})\, u(x,t):=P.V\int_{\ren} \,\dfrac{|u(x,t)-u(y,t)|^{p-2}(u(x,t)-u(y,t))}{|x-y|^{N+ps}} \,dy$$
with $1<p<N, s\in (0,1)$ and $f, u_0$ are non negative functions.
The main goal of this work is to discuss the existence of solution according to the values of $p$ and $\l$.
\end{abstract}
\section{Introduction.}\label{sec:s0}
This paper deals with the following parabolic problem
\begin{equation}\label{eq:def}
\left\{
\begin{array}{rcll}
u_t+(-\D^s_{p}) u&=& \dyle \l \frac{u^{p-1}}{|x|^{ps}}  &
\text{ in } \O_{T}=\Omega \times (0,T)  , \\
u&\ge& 0 &
\text{ in }\Omega, \\
u&=&0 & \text{ in }(\ren\setminus\O) \times (0,T), \\
u(x,0)&=&u_0(x)& \mbox{  in  }\O,
\end{array}%
\right.
\end{equation}
where $\O$ is a bounded domain and
$$ (-\D^s_{p})\, u(x,t):=P.V\int_{\mathbb{R}^{N}} \,
\frac{|u(x,t)-u(y,t)|^{p-2}(u(x,t)-u(y,t))}{|x-y|^{N+ps}}\ dy,$$
is the non local $p$ laplacian operator. Problem \eqref{eq:def} is related to the following Hardy-Sobolev inequality .
\begin{Theorem} \label{S-Hardy}(Fractional Hardy-Sobolev  inequality)
Let $N>1$ and $0<s<1$. Assume that $1\le p<\frac{N}{s}$, then for all $u \in W^{s,p}(\ren)$
we have
\begin{equation}\label{hardy}
\frac 12 \iint_{\mathbb{R}^{2N}}
\dfrac{|u(x)-u(y)|^p}{|x-y|^{N+ps}}dxdy\ge \L_{N,p,s}\irn
\dfrac{|u(x)|^p}{|x|^{ps}}dx
\end{equation}
where the constant $\L_{N,p,s}$ is given by
\begin{equation}\label{LL}
\L_{N,p,s}=\int_0^1
\s^{ps-1}|1-\sigma^{\frac{N-ps}{p}}|^{p}K(\sigma)d\s
\end{equation}
and
\begin{equation}\label{kkk}
K(\sigma)=\dint\limits_{|y'|=1}\dfrac{dH^{n-1}(y')}{|x'-\s
y'|^{N+ps}}=2\frac{\pi^{\frac{N-1}{2}}}{\beta(\frac{N-1}{2})}\int_0^\pi
\frac{\sin^{N-2}(\xi)}{(1-2\sigma \cos
(\xi)+\sigma^2)^{\frac{N+ps}{2}}}d\xi.
\end{equation}
The constant $\L_{N,p,s}$ is optimal and not achieved.
\end{Theorem}
We refer to \cite{FS} and the references therein for the  proof. See too \cite{AB} and \cite{AM}.

For $p=2$ and $s=1$ the problem \eqref{eq:def} was studied in \cite{BG}.  The authors proved existence and
nonexistence results related to the fact that $\l\le \L_{N,2}$ or
$\l>\L_{N,2}$, respectively. The nonlocal case  has been studied  in \cite{AMPP}. The authors by proving a suitable Harnack inequality, analyzed the optimal relation between integrability of the data and the spectral value $\lambda$. Moreover they proved the existence of a critical exponent $q_+(\lambda,s)$ depending only on $\l$ such that existence holds for a semilinear problem if and only if the power  $q<q_+(\lambda)$.

For $p\neq 2$ and $s =1$, the problem was first widely analyzed in \cite{GP}. In \cite{AP}, the authors studied  some qualitative and quantitative properties of the weak solutions. In \cite{DGP}, the authors studied a more general class of operator and in particular complete the previous study showing that  if $\frac{2N}{N+1}\le p<2$, the problem has a distributional solution far from the origin. This fact was proved using a class of the Caffarelli-Khon-Nirenberg inequality that holds for any degenerate radial potential in the local case, see \cite{CKN}.
We quote here the recent result in \cite{AABP} where the authors proved the existence of entropy solution for all data in $L^1$ but without the Hardy potential. Such problem has not finite speed of propagation property, that can be immediately extended to problem \eqref{eq:def}.

To study the problem \eqref{eq:def}  in the fractional setting, $s<1$, there appear some challenging difficulties with respect to the local case, that must be solved. Precisely the fractional version of some local results in \cite{DGP} need a deep analysis in the nonlocal framework to reach results on existence.

The paper is organized as follows. In Section \ref{sec2} we give some auxiliary results
related to fractional Sobolev spaces and some functional inequalities. We present also some algebraic inequalities that will be used to overcame the lost of the possibility of \textit{integration by part} for the nonlocal operator. In an Appendix, we give a detailed proof of this algebraic inequality.

To deal with the case $\frac{2N}{N+s}\le p<2$ and $\l>\L_{N,p,s}$, as it was proved in \cite{DGP} in the local case, we need to consider fractional Sobolev spaces with very degenerate potential. In this case and as it was observed in \cite{AB}, on the contrary to the case of singular potential, we need to use a new approach to define the fractional Sobolev spaces. Hence in subsection \ref{degenerate} we define such natural spaces where the solution will live and we give some connection with the spaces defined in \cite{AB}.

In Section \ref{sec3} we will consider the case $\l\le \L_{N,p,s}$,
in this case we prove the existence of a global solution that is in a suitable energy space.

The case $\l>\L_{N,p,s}$ and $p<2$ is studied in Section \ref{singular}. According to the value of $p$, we prove the existence of a solution that is in a suitable fractional Sobolev space.  If $\frac{2N}{N+s}\le p<2$, that is the more delicate case, we are able to prove the existence of a solution far from the origin which is, modulo a suitable weight, in a fractional weighted Sobolev space.

The question of extinction in finite time is analyzed in Section \ref{singg}.
According to smallness condition on $u_0$, we prove the finite time extension properties. The same property is proved if we add \textit{a concave} potential of $u$ as a reaction term in \eqref{eq:def}.

In the last section we consider the case $p>2$ and $\l>\L_{N,p,s}$. Since the finite speed propagation properties does not hold in the nonlocal case, we are able to show that problem \eqref{eq:def} has non nonnegative solution in an appropriate sense. This result can be extended to large class of nonlinearities and can be seen as non local version of the results obtained in \cite{GP} and \cite{AP1}.
\section{Preliminaries and functional setting}\label{sec2}

Let us begin by stating some preliminaries tools about fractional Sobolev spaces and their properties that we will use systematically in this paper.
We refer to \cite{DPV} and \cite{Adams} for more details.

Assume that $s\in (0,1)$ and $p>1$. Let $\O\subset \ren$, then the fractional Sobolev
spaces $W^{s,p}(\Omega)$, is defined by
$$
W^{s,p}(\Omega)\equiv
\Big\{ \phi\in
L^p(\O):\int_{\O}\dint_{\O}|\phi(x)-\phi(y)|^pd\nu<+\infty\Big\}
$$
where $d\nu=\dyle\frac{dxdy}{|x-y|^{N+ps}}$. It is clear that $W^{s,p}(\O)$ is a Banach space endowed with the following norm
$$
\|\phi\|_{W^{s,p}(\O)}=
\Big(\dint_{\O}|\phi(x)|^pdx\Big)^{\frac 1p}
+\Big(\dint_{\O}\dint_{\O}|\phi(x)-\phi(y)|^pd\nu\Big)^{\frac
1p}.
$$
In the same way we define the space $W^{s,p}_{0} (\O)$ as
the completion of $\mathcal{C}^\infty_0(\O)$ with respect to the
previous norm.

If $\O$ is bounded regular domain, we
can endow $W^{s,p}_{0}(\O)$ with the equivalent norm
$$
||\phi||_{W^{s,p}_{0}(\O)}=
\Big(\int_{\O}\dint_{\O}\dfrac{|\phi(x)-\phi(y)|^p}{|x-y|^{N+ps}}{dxdy}\Big)^{\frac
1p}.
$$

The next Sobolev inequality is proved in \cite{DPV}.
\begin{Theorem} \label{Sobolev}(Fractional Sobolev inequality)
Assume that $0<s<1, p>1$ satisfy $ps<N$. There exists a positive constant $S\equiv S(N,s,p)$ such that for all
$v\in C_{0}^{\infty}(\ren)$,
$$
\iint_{\re^{2N}}
\dfrac{|v(x)-v(y)|^{p}}{|x-y|^{N+ps}}\,dxdy\geq S
\Big(\dint_{\mathbb{R}^{N}}|v(x)|^{p_{s}^{*}}dx\Big)^{\frac{p}{p^{*}_{s}}},
$$
where $p^{*}_{s}= \dfrac{pN}{N-ps}$.
\end{Theorem}

To treat the case $\l=\L_{N,p,s}$, we need the next improved Hardy-Sobolev inequality obtained in \cite{AB} and \cite{AM}.
\begin{Theorem}\label{impro}
Let $p>1$, $0<s<1$ and $N>ps$. Assume that $\Omega\subset
\mathbb{R}^N$ is a bounded domain containing the origin, then for all $1<q<p$, there
exists a positive constant $C=C(\Omega, q, N, s)$ such that for
all $u\in \mathcal{C}_0^\infty(\Omega)$,
\begin{equation}\label{sara}
\frac 12\iint_{\mathbb{R}^{2N}}\,
\frac{|u(x)-u(y)|^{p}}{|x-y|^{N+ps}}\,dx\,dy - \Lambda_{N,p,s}
\int_{\mathbb{R}^{N}} \frac{|u(x)|^p}{|x|^{ps}}\,dx\geq C
\int_{\Omega}\int_{\Omega}\frac{|u(x)-u(y)|^p}{|x-y|^{N+qs}}\,dx\,dy.
\end{equation}
\end{Theorem}

Now, for $w\in W^{s,p}(\ren)$, we set
$$
(-\Delta)^s_{p} w(x)={ P.V. }
\dint_{\ren}\dfrac{|w(x)-w(y)|^{p-2}(w(x)-w(y))}{|x-y|^{N+ps}}{dy}.
$$
It is clear that for all $w, v\in W^{s,p}(\ren)$, we have
$$
\langle (-\Delta)^s_{p}w,v\rangle
=\dfrac 12\iint_{\re^{2N}}\dfrac{|w(x)-w(y)|^{p-2}(w(x)-w(y))(v(x)-v(y))}{|x-y|^{N+ps}}{dxdy}.
$$
If $w, v\in W^{s,p}_0(\O)$, we have
$$
\langle (-\Delta)^s_{p}w,v\rangle
=\dfrac 12\iint_{D_\O}\dfrac{|w(x)-w(y)|^{p-2}(w(x)-w(y))(v(x)-v(y))}{|x-y|^{N+ps}}{dxdy},
$$
where $D_{\O}=\ren\times \ren\setminus \mathcal{C}\O\times
\mathcal{C}\O$.

The next Picone inequality will be useful to prove the non existence result for $p>2$.
\begin{Theorem}(Picone inequality)\label{pic}
Let $w\in W^{s,p}_0(\O)$ be such that $w>0$ in $\O$. Assume
that $(-\Delta)^s_{p}w = \nu$ with $\nu\in L^1_{loc}(\ren)$ and
$\nu\gneqq 0$, then for all $\psi\in \mathcal{C}^\infty_0(\O)$, we
have
$$
\frac 12
\iint_{D_\O}\dfrac{|\psi(x)-\psi(y)|^{p}}{|x-y|^{N+ps}}dx\,dy\ge \io
\frac{(-\Delta)^s_{p}w }{w^{p-1}}|\psi|^p dx.
$$
\end{Theorem}
We refer to \cite{BPV} and \cite{AB} for a complete proof and other application of the Picone inequality.

We define now the corresponding parabolic spaces.

The space $L^{p}(0,T; W^{s,p}_0(\O))$ is defined as the set of function $\phi$ such that
$\phi\in L^p(\O_T)$ with $||\phi||_{L^{p}(0,T; W^{s,p}_0(\O))}<\infty$ where
$$
||\phi||_{L^{p}(0,T; W^{s,p}_0(\O))}=\Big(\int_0^T\iint_{D_{\O}}|\phi(x,t)-\phi(y,t)|^pd\nu\,dt\Big)^{\frac
1p}.
$$
It is clear that $L^{p}(0,T; W^{s,p}_0(\O))$ is a Banach spaces.

In the case where the data $(f,u_0)\in L^2(\O_T)\times L^2(\O)$, then we can deal with energy solution, more precisely we have the next definition.
\begin{Definition}\label{energy}
Assume  $(f,u_0)\in L^2(\O_T)\times L^2(\O)$, then we say that $u$ is an energy solution to problem \eqref{eq:def} if $u\in L^{p}(0,T; W^{s,p}_0(\O))\cap \mathcal{C}([0,T], L^p(\O))$, $u_t\in L^{p'}(0,T; W^{-s,p'}_0(\O))$, where $W^{-s,p'}_0(\O)$ is the dual space of $W^{s,p}_0(\O)$ and for all $v\in  L^{p}(0,T; W^{s,p}_0(\O))$ we have
\begin{equation*}
\begin{array}{lll}
&\dyle\int_0^T\langle u_t, v\rangle dt +\dfrac 12\int_0^T\iint_{D_{\O}}U(x,y,t)(v(x,t)-v(y,t))d\nu\ dt\\
&\dyle=\l\iint_{\O_T} \frac{|u|^{p-2}u}{|x|^{ps}}v dx\,dt
\end{array}
\end{equation*}
and $u(x,.)\to u_0$ strongly in $L^2(\O)$ as $t\to 0$ where
$$
U(x,y,t)\equiv |u(x,t)-u(y,t)|^{p-2}(u(x,t)-u(y,t)).
$$
\end{Definition}
Notice that the existence of energy solution follows using classical argument for monotone operator as in \cite{Lio}.

Before closing this section, we  recall some useful algebraic inequalities which will be used throughout the paper. The proof
follows using suitable rescaling arguments.
\begin{Lemma}\label{algg}
Assume that $p\ge 1$, $(a, b) \in (\re^+)^2$ and $\a>0$, then there exist $c_1, c_2,c_3, c_4>0$, such that
\begin{equation}\label{alge1}
(a+b)^\a\le c_1a^\a+c_2b^\a,
\end{equation}
and
\begin{equation}\label{alge3}
|a-b|^{p-2}(a-b)(a^{\a}-b^{\a})\ge c_3|a^{\frac{p+\a-1}{p}}-b^{\frac{p+\a-1}{p}}|^p.
\end{equation}
In the case where $\a\ge 1$, then under the same conditions on $a,b,p$ as above, we have
\begin{equation}\label{alge2}
|a+b|^{\a-1}|a-b|^{p}\le c_4 |a^{\frac{p+\a-1}{p}}-b^{\frac{p+\a-1}{p}}|^p.
\end{equation}
\end{Lemma}
The next algebraic inequality is new and can be seen as an extension of the {\it integration by part formula} when using a product as a test function in the local case. The proof is given in the Appendix.
\begin{Lemma}\label{real11}
There exist two positive constants $C_1<1<C_2$ such that for all $a_1,a_2\in \re$ and for all $b_1,b_2\ge 0$, we have
\begin{equation}\label{alge4}
|a_1-a_2|^{p-2}(a_1-a_2)(a_1b_1-a_2b_2)\ge C_1 |a_1b^{\frac{1}{p}}_1-a_2b^{\frac{1}{p}}_2|^p-C_2(\max\{|a_1|, |a_2|\})^p|b^{\frac{1}{p}}_1-b^{\frac{1}{p}}_2|^p.
\end{equation}
\end{Lemma}

\subsection{Fractional Sobolev space associated to degenerate potential}\label{degenerate}
To analyze the regularity of solution to problem \eqref{eq:def} when $p<2$, we need to develop some weighted Sobolev type inequalities with degenerate potential. In the local case and as it was proved in \cite{DGP}, this type of estimate was a consequence of the well known Caffarelli-Khon-Nirenberg inequalities proved for a large class of weights that cover all the radial degenerate potentials.

As in \cite{AB} and \cite{AM}, setting
$$
W^{s,p}_\b(\ren)\dyle := \Big\{ \phi\in
L^p(\ren,\frac{dx}{|x|^{2\beta}}):\iint_{\re^{2N}}\dfrac{|\phi(x)-\phi(y)|^p}{|x-y|^{N+ps}}\dfrac{dxdy}{|x|^\beta|y|^\beta}<+\infty\Big\},
$$
then for $-ps<\beta<\frac{N-ps}{2}$, the space $W^{s,p}_\b(\ren)$ is a Banach space endowed with the norm
$$
\|\phi\|_{W^{s,p}_\beta(\ren)}=
\Big(\dint_{\ren}\frac{|\phi(x)|^pdx}{|x|^{2\beta}}\Big)^{\frac 1p}
+\Big(\iint_{\re^{2N}}\dfrac{|\phi(x)-\phi(y)|^p}{|x-y|^{N+ps}}\dfrac{dxdy}{|x|^\beta|y|^\beta}\Big)^{\frac
1p}.
$$
As a consequence the next weighted Sobolev inequality is proved.
\begin{Theorem} \label{Sobolev1}(Weighted fractional Sobolev inequality)
Assume that $0<s<1$ and $p>1$ are such that $ps<N$. Let
$-ps<\beta<\dfrac{N-ps}{2}$, then there exists a positive constant $S(N,s,\beta)$ such that for all
$v\in C_{0}^{\infty}(\ren)$,
\begin{equation}\label{sobolev00}
\dint_{\mathbb{R}^{N}}\dint_{\mathbb{R}^{N}}
\dfrac{|v(x)-v(y)|^{p}}{|x-y|^{N+ps}}\frac{dx}{|x|^{\beta}}\,
\frac{dy}{|y|^{\beta}}\geq S(N,s,\beta)
\Big(\dint_{\mathbb{R}^{N}}
\dfrac{|v(x)|^{p_{s}^{*}}}{|x|^{2\beta\frac{p_{s}^{*}}{p}}}\Big)^{\frac{p}{p^{*}_{s}}},
\end{equation}
where $p^{*}_{s}= \dfrac{pN}{N-ps}$.
\end{Theorem}
Hence we can define $D^{s,p}_\beta(\ren)$ as
the completion of $\mathcal{C}^\infty_0(\ren)$ with respect to the norm
$$
\|\phi\|_{D^{s,p}_\beta(\ren)}=\Big(\iint_{\re^{2N}}\dfrac{|\phi(x)-\phi(y)|^p}{|x-y|^{N+ps}}\dfrac{dxdy}{|x|^\beta|y|^\beta}\Big)^{\frac
1p}.
$$
It is clear that the following weighted Hardy inequality
	\begin{equation}\label{IngL1}
	2\Upsilon(\g)\dint\limits_{\re^N} \dfrac{|u(x)|^p}{|x|^{ps+2\beta}}\,dx
\leq\dyle \iint_{\re^{2N}}\dfrac{|u(x)-u(y)|^p}{|x-y|^{N+ps}}\dfrac{dx}{|x|^{\beta}}
	\dfrac{dy}{|y|^\beta},
		\end{equation}
	where
\begin{equation}\label{ga}
\Upsilon(\g)=\dint\limits_1^{+\infty}K(\sigma)(\sigma^\g-1)^{p-1}\left(\sigma^{N-1-\beta-\g(p-1)}-\sigma^{\beta+ps-1}\right)\,d\sigma.
\end{equation}
holds for all $u\in D^{s,p}_\beta(\ren)$.
Notice that $\Upsilon(\g)$ is well defined if $-ps<\beta<\frac{N-ps}{2}$. In this case $\Upsilon(\g)>0$.

If $\beta\le -sp$, then $\mathcal{C}^\infty_0(\ren)\nsubseteqq W^{s,p}_\beta(\ren)$. To see that we fix $\phi\in \mathcal{C}^\infty_0(B_{4}(0))$ such that $0\le \phi\le 1$ and $\phi=1$ in $B_1(0)$, then
\begin{eqnarray*}
\dyle \iint_{\mathbb{R}^{2N}} \dfrac{|\phi(x)-\phi(y)|^p}{|x-y|^{N+ps}}\dfrac{dxdy}{|x|^\beta|y|^\beta} &\ge & \dint_{\ren\backslash B_{4}(0)}\dint_{B_1(0)}\dfrac{1}{|x-y|^{N+ps}}\dfrac{dxdy}{|x|^\beta|y|^\beta}\\
&\ge & \dint_{\ren\backslash B_{4}(0)}\dfrac{1}{(|y|+4)^{N+ps}}\dfrac{dy}{|y|^\beta}
\dint_{B_1(0)}\dfrac{dx}{|x|^\beta}\\
&\ge & C(N-\beta)\dint_{\ren\backslash B_{4}(0)}\dfrac{1}{(|y|+4)^{N+ps}}\dfrac{dy}{|y|^\beta}.
\end{eqnarray*}
Since $\beta\le -ps$, then $\dint_{\ren\backslash B_{4}(0)}\dfrac{1}{(|y|+4)^{N+ps}}\dfrac{dy}{|y|^\beta}=\infty$.

\

Thus to deal with degenerate weight we need to adapt new approach.

Let $-\infty<\a<\frac{N-ps}{2}$ and define the space
$$E_\a(\ren)=\bigg\{u: |x|^\a u\in D^{s,p}(\re^N)\,\,\text{ i.e:} \iint_{\re^{2N}}\dfrac{||x|^\a u(x)-|y|^\a u(y)|^p}{|x-y|^{N+ps}}dxdy< \infty \bigg\}.$$
Using The classical Sobolev inequality we conclude that $E_\a(\ren)$ is a Banach space and
$$
S
\Big(\dint_{\mathbb{R}^{N}}|u(x)|^{p_{s}^{*}}|x|^{p^*_s\a} \Big)^{\frac{p}{p^{*}_{s}}}\le \iint_{\re^{2N}}\dfrac{||x|^\a u(x)-|y|^\a u(y)|^p}{|x-y|^{N+ps}}dxdy,
$$
that can be seen as a Caffarelli-Khon-Nirenberg inequality.

The main result of this subsection is the following.
\begin{Theorem}
Assume that $-ps<\beta<\frac{N-ps}{2}$, then $W^{s,p}_\b(\re^N)=E_\a(\ren)$ with $\a=-\dfrac{2\b}{p}$.
\end{Theorem}
\begin{proof}
To prove the main result we have just to show the existence of $C_1,C_2>0$ such that for all $u\in \mathcal{C}^\infty_0(\ren)$, we have
\begin{equation}\label{spa}
C_1\|u\|_{E_\a(\ren)}\le \|u\|_{W^{s,p}_\beta(\ren)}\le C_2\|u\|_{E_\a(\ren)}.
\end{equation}
Let us begin by proving the first inequality. In this case the proof follows using closely the computations in \cite{AB}. For the reader convenience we include here all details. In what follows, we denote by $C_1, C_2,... $ any positive constants that are independent of $u$ and can change from one line to another.

Define
\begin{equation}\label{vvv}
w(x)=|x|^{-\alpha}=|x|^{\frac{2\b}{p}}, v(x)= \dfrac{u(x)}{w(x)},
\end{equation}
then
$$
\dfrac{|u(x)-u(y)|^p}{|x-y|^{N+ps}|x|^\beta|y|^\beta}=\dfrac{|u(x)-u(y)|^p}{|x-y|^{N+ps}}\, \dfrac{1}{(w(x)w(y))^{\frac{p}{2}}}
$$
and
$$
\begin{array}{rcl}
\dfrac{|u(x)-u(y)|^p}{|x-y|^{N+ps}}\, \dfrac{1}{(w(x)w(y))^{\frac{p}{2}}} &= & \dfrac{\big|(v(x)-v(y))-v(y)w(y)(\dfrac{1}{w(x)}-\dfrac{1}{w(y)})\big|^p}{|x-y|^{N+ps}}\left(\dfrac{w(x)}{w(y)}\right)^{\frac{p}{2}}\\
&\equiv & f_{1}(x,y).
\end{array}
$$
In the same way we have
$$
\begin{array}{rcl}
\dfrac{|u(x)-u(y)|^p}{|x-y|^{N+ps}}\, \dfrac{1}{(w(x)w(y))^{\frac{p}{2}}} & = & \dfrac{\big|(v(y)-v(x))-v(x)w(x)(\dfrac{1}{w(y)}-\dfrac{1}{w(x)})\big|^p}{|x-y|^{N+ps}}\left(\dfrac{w(y)}{w(x)}\right)^{\frac{p}{2}}\\
&\equiv &  f_2(x,y).
\end{array}
$$
It is clear that
$$\dyle\iint_{\re^{2N}}f_1(x,y)\,dx\,dy=\iint_{\re^{2N}}f_2(x,y)\,dx\,dy,$$
and
	$$\iint_{\re^{2N}}\dfrac{|u(x)-u(y)|^p}{|x-y|^{N+ps}}\dfrac{dx}{|x|^\b}\dfrac{dy}{|x|^\b}=
	\frac 12\iint_{\re^{2N}}f_1(x,y)\,dx\,dy+\frac 12\iint_{\re^{2N}}f_2(x,y)\,dx\,dy.$$
Since
\begin{eqnarray*}
&f_{1}(x,y)\geq
\left(\dfrac{w(x)}{w(y)}\right)^{\frac{p}{2}} \times\\
& \Big[ \dfrac{|v(x)-v(y)|^p}{|x-y|^{N+ps}}-p
\dfrac{|v(x)-v(y)|^{p-2}}{|x-y|^{N+ps}}\big\langle
v(x)-v(y),
w(y)v(y)(\dfrac{1}{w(x)}-\dfrac{1}{w(y)})\big\rangle\\
&+C(p)\dfrac{|w(y)v(y)(\dfrac{1}{w(x)}-\dfrac{1}{w(y)})
|^p}{|x-y|^{N+ps}}\Big],
\end{eqnarray*}
using Young inequality, it holds that
\begin{eqnarray*}
&f_{1}(x,y)\geq \left(\dfrac{w(x)}{w(y)}\right)^{\frac{p}{2}} \times\\
& \Big[ C_1\dfrac{|v(x)-v(y)|^p}{|x-y|^{N+ps}}
-C_2\dfrac{|w(y)v(y)(\dfrac{1}{w(x)}-\dfrac{1}{w(y)})|^{p}}{|x-y|^{N+ps}}|\Big]
\end{eqnarray*}
with $C_1,C_2>0$ independent of $u$. In a symmetric way, we reach that 	
	\begin{eqnarray*}
		&f_{2}(x,y)\geq \left(\dfrac{w(y)}{w(x)}\right)^{\frac{p}{2}} \times\\
		& \Big[ C_1\dfrac{|v(x)-v(y)|^p}{|x-y|^{N+ps}}
		-C_2\dfrac{|w(x)v(x)(\dfrac{1}{w(y)}-\dfrac{1}{w(x)})|^{p}}{|x-y|^{N+ps}}|\Big].
	\end{eqnarray*}
	
	Thus we get the existence of positive constants $C_1, C_2, {C_3}$
	such that
	\begin{eqnarray*}
		&\dyle \iint_{\re^{2N}}
		\dfrac{|u(x)-u(y)|^p}{|x-y|^{N+ps}}\dfrac{dx}{|x|^{\beta}}
		\dfrac{dy}{|y|^\beta}\ge {C_1}\iint_{\re^{2N}}
		\dfrac{|v(x)-v(y)|^p}{|x-y|^{N+ps}}\Big[\left(\dfrac{w(y)}{w(x)}\right)^{\frac{p}{2}}
		+\left(\dfrac{w(x)}{w(y)}\right)^{\frac{p}{2}}\Big]dxdy\\
		&-\dyle {C_2}\iint_{\re^{2N}}\left(\dfrac{w(x)}{w(y)}\right)^{\frac{p}{2}}\dfrac{|w(y)v(y)(\dfrac{1}{w(x)}-\dfrac{1}{w(y)})|^{p}}{|x-y|^{N+ps}}|dxdy\\
		&-\dyle{C_3}\iint_{\re^{2N}}\left(\dfrac{w(y)}{w(x)}\right)^{\frac{p}{2}}\dfrac{|w(x)v(x)(\dfrac{1}{w(y)}-\dfrac{1}{w(x)})|^{p}}{|x-y|^{N+ps}}dx\,dy.
	\end{eqnarray*}
	Since $
	\Big[\left(\dfrac{w(y)}{w(x)}\right)^{\frac{p}{2}}
	+\left(\dfrac{w(x)}{w(y)}\right)^{\frac{p}{2}}\Big]\ge 1,
	$
	we get
	\begin{equation}\label{new4}
	\begin{array}{lll}
	&\dyle \iint_{\re^{2N}}
	\dfrac{|v(x)-v(y)|^p}{|x-y|^{N+ps}}dxdy\le
	C_1\iint_{\re^{2N}}
	\dfrac{|u(x)-u(y)|^p}{|x-y|^{N+ps}}\dfrac{dx}{|x|^{\beta}}
	\dfrac{dy}{|y|^\beta}\\
	&+\dyle C_2\iint_{\re^{2N}}\left(\dfrac{w(x)}{w(y)}\right)^{\frac{p}{2}}\dfrac{|w(y)v(y)(\dfrac{1}{w(x)}-\dfrac{1}{w(y)})|^{p}}{|x-y|^{N+ps}}|dxdy\\
	&+C_3\dyle\iint_{\re^{2N}}\left(\dfrac{w(y)}{w(x)}\right)^{\frac{p}{2}}\dfrac{|w(x)v(x)(\dfrac{1}{w(y)}-\dfrac{1}{w(x)})|^{p}}{|x-y|^{N+ps}}dxdy.
	\end{array}
	\end{equation}
Define
$$
g_1(x,y)=\left(\dfrac{w(y)}{w(x)}\right)^{\frac{p}{2}}\dfrac{|w(x)v(x)(\dfrac{1}{w(y)}-\dfrac{1}{w(x)})|^{p}}{|x-y|^{N+ps}}
$$
and
$$
g_2(x,y)=\left(\dfrac{w(x)}{w(y)}\right)^{\frac{p}{2}}\dfrac{|w(y)v(y)(\dfrac{1}{w(x)}-\dfrac{1}{w(y)})|^{p}}{|x-y|^{N+ps}},
$$
then
$ \dyle \iint_{\re^{2N}} g_1(x,y)dxdy=\iint_{\re^{2N}} g_2(x,y)dxdy.$ Hence we have just to estimate the first integral. Taking into consideration the definition of $v$ and $w$ given in \eqref{vvv}, it holds
$$
g_1(x,y)=\dfrac{|u(x)|^p\Big||x|^{\frac{2\beta}{p}}-|y|^{\frac{2\beta}{p}}\Big|^p}{|x|^{3\beta}|y|^{\beta}|x-y|^{N+ps}}.
$$
Therefore, we get
\begin{eqnarray*}
I\equiv \iint_{\re^{2N}} g_1(x,y)dxdy &=&
\dint\limits_{\re^N}\dfrac{|u(x)|^p}{|x|^{3\beta}}
\Big(\dint\limits_{\re^N}\dfrac{\Big||x|^{\frac{2\beta}{p}}-|y|^{\frac{2\beta}{p}}\Big|^p}{|y|^{\beta}|x-y|^{N+ps}}dy\Big)dx.
\end{eqnarray*}
Now, we follow closely the radial computations as \cite{FV} and \cite{G}. We set $r=|x|$ and $\rho=|y|$, then $x=rx', y=\rho y'$ with $|x'|=|y'|=1$, thus
	\begin{eqnarray*}
I=\dint\limits_{\re^N}\dfrac{|u(x)|^p}{|x|^{3\beta}} \Big[
\dint\limits_0^{+\infty}\dfrac{|r^{\frac{2\beta}{p}}-\rho^{\frac{2\beta}{p}}|^p\rho^{N-1}}{\rho^{\beta}}\left(
\dint\limits_{|y'|=1}\dfrac{dH^{n-1}(y')}{|r x'-\rho y'|^{N+ps}}
\right) \,d\rho\Big]dx.
\end{eqnarray*}
Let $\s=\frac{\rho}{r}$, hence
\begin{eqnarray*}
I=\dint\limits_{\re^N}\dfrac{|u(x)|^p}{|x|^{2\beta+ps}}
\Big[
\dint\limits_0^{+\infty}|1-\s^{\frac{2\beta}{p}}|^p\s^{N-1-\beta}K(\s)\,d\s\Big]dx.
\end{eqnarray*}
Notice that $K(\frac{1}{\xi})=\xi^{N+ps}K(\xi)$ for $\xi>0$, thus
$$\dint\limits_0^{+\infty}|1-\s^{\frac{2\beta}{p}}|^p\s^{N-1-\beta}K(\s)\,d\s=
\dint_1^\infty(\s^{\frac{2\beta}{p}}-1)^p(\s^{N-1-\beta}+\s^{ps-1-\beta})K(\s)\,d\s
$$

Taking into consideration the behavior of $K$ near $1$ and $\infty$, we conclude that  $$\dint\limits_0^{+\infty}|1-\s^{\frac{2\beta}{p}}|^p\s^{N-1-\beta}K(\s)\,d\s= C_3\in (0, \infty).$$ Thus
	$$
	\iint_{\re^{2N}}
	g_1(x,y)dxdy=C_3\dint\limits_{\re^N}
	\dfrac{|u(x)|^p}{|x|^{2\beta+ps}}dx.
	$$
	Now, using the fractional weighted Hardy inequality given in \eqref{IngL1}, we reach that
	\begin{equation}\label{new3}
	\iint_{\re^{2N}} g_1(x,y)dxdy \leq C_4
	\iint_{\re^{2N}}
	\dfrac{|u(x)-u(y)|^p}{|x-y|^{N+ps}}\dfrac{dx}{|x|^{\beta}}
	\dfrac{dy}{|y|^\beta}.
	\end{equation}
	Combining  \eqref{new3} and
	\eqref{new4}, it holds
	\begin{equation}
	\iint_{\re^{2N}}
	\dfrac{|v(x)-v(y)|^p}{|x-y|^{N+ps}}dxdy\le
	C_1\iint_{\re^{2N}}
	\dfrac{|u(x)-u(y)|^p}{|x-y|^{N+ps}}\dfrac{dx}{|x|^{\beta}}
	\dfrac{dy}{|y|^\beta}
	\end{equation}
and then $\tilde{C}\|u\|_{E_\a(\ren)}\le \|u\|_{W^{s,p}_\beta(\ren)}$ with $\tilde{C}>0$.

\

We deal now with the second inequality in \eqref{spa}.
Notice that
\begin{eqnarray*}
&\dfrac{|v(x)-v(y)|^p}{|x-y|^{N+ps}}=\frac{1}{2} \dfrac{\big|(u(y)-u(x))-\dfrac{u(x)}{w(x)}(w(y)-w(x))\big|^p}{|x-y|^{N+ps}}\Big[\dfrac{1}{(w(x))^{p}}+\dfrac{1}{(w(y))^{p}}\Big]\\
& \ge C_1\dfrac{|u(x)-u(y)|^p}{|x-y|^{N+ps}}
\Big[\dfrac{1}{(w(x))^{p}}+\dfrac{1}{(w(y))^{p}}\Big]\\
&-{C_2}\left(\dfrac{1}{w(y)}\right)^{p} \dfrac{|w(y)-w(x)|^{p}}{|x-y|^{N+ps}}|\dfrac{u(x)}{w(x)}|^p-{C_3}\left(\dfrac{1}{w(x)}\right)^{p} \dfrac{|w(x)-w(y)|^{p}}{|x-y|^{N+ps}}|\dfrac{u(y)}{w(y)}|^p.
\end{eqnarray*}
Since
$$\Big[\dfrac{1}{(w(x))^{p}}+\dfrac{1}{(w(y))^{p}}\Big]\ge \dfrac{1}{(w(x))^{\frac p2}}\dfrac{1}{(w(y))^{\frac p2}}\equiv \dfrac{1}{|x|^{\b}|y|^{\b}},$$
then
\begin{eqnarray*}
&\dfrac{|v(x)-v(y)|^p}{|x-y|^{N+ps}}\ge C_1\dfrac{|u(x)-u(y)|^p}{|x-y|^{N+ps}|x|^{\b}|y|^{\b}}\\
&-{C_2}\left(\dfrac{1}{w(y)}\right)^{p} \dfrac{|w(y)-w(x)|^{p}}{|x-y|^{N+ps}}|\dfrac{u(x)}{w(x)}|^p-{C_3}\left(\dfrac{1}{w(x)}\right)^{p} \dfrac{|w(x)-w(y)|^{p}}{|x-y|^{N+ps}}|\dfrac{u(y)}{w(y)}|^p.
\end{eqnarray*}
Therefore we conclude that
\begin{equation}\label{RR}
\begin{array}{lll}
& C_1\dyle \iint_{\re^{2N}}\dfrac{|u(x)-u(y)|^p\:dxdy}{|x-y|^{N+ps}|x|^{\b}|y|^{\b}}\le \dint_{\re^N}\dint_{\re^N} \dfrac{|v(x)-v(y)|^p}{|x-y|^{N+ps}}dxdy\\ \\
&+{C_2}\dyle \iint_{\re^{2N}}\left(\dfrac{1}{w(y)}\right)^{p} \dfrac{|w(y)-w(x)|^{p}}{|x-y|^{N+ps}}|\dfrac{u(x)}{w(x)}|^pdxdy\\ \\ &+\dyle {C_3}\iint_{\re^{2N}}\left(\dfrac{1}{w(x)}\right)^{p} \dfrac{|w(x)-w(y)|^{p}}{|x-y|^{N+ps}}|\dfrac{u(y)}{w(y)}|^pdxdy.
\end{array}
\end{equation}
As in the first case, setting
$$
\tilde{g}_1(x,y)=\left(\dfrac{1}{w(y)}\right)^{p}|\dfrac{u(x)}{w(x)}|^p \dfrac{|w(y)-w(x)|^{p}}{|x-y|^{N+ps}}
$$
and
$$
\tilde{g}_2(x,y)=\left(\dfrac{1}{w(x)}\right)^{p}|\dfrac{u(y)}{w(y)}|^p \dfrac{|w(x)-w(y)|^{p}}{|x-y|^{N+ps}},$$
it holds that  $\dyle\iint_{\re^{2N}} \tilde{g}_1(x,y)dxdy=\iint_{\re^{2N}} \tilde{g}_2(x,y)dxdy$
and
\begin{eqnarray*}
\iint_{\re^{2N}} \tilde{g}_1(x,y)dxdy &=&\dint\limits_{\re^N}\dfrac{|u(x)|^p}{|x|^{2\beta}}
			\Big(\dint\limits_{\re^N}\dfrac{\Big||x|^{\frac{2\beta}{p}}-|y|^{\frac{2\beta}{p}}\Big|^p}{|y|^{2\beta}|x-y|^{N+ps}}dy\Big)dx.
		\end{eqnarray*}
Thus, as in the first case,
\begin{eqnarray*}
\iint_{\re^{2N}} \tilde{g}_1(x,y)dxdy =C_4\dint\limits_{\re^N}\dfrac{|u(x)|^p}{|x|^{2\beta+ps}}dx=C_4\dint\limits_{\re^N}\dfrac{|v(x)|^p}{|x|^{ps}}dx
\end{eqnarray*}
where $C_4=\dint_0^{+\infty}|1-\s^{\frac{2\beta}{p}}|^p\s^{N-1-2\beta}K(\s)\,d\s=C<\infty.$
Now, using fractional Hardy inequality in Theorem \ref{S-Hardy}, for $v$, we get
		\begin{equation}\label{new03}
		\iint_{\re^{2N}} g_1(x,y)dxdy \leq C_5
		\iint_{\re^{2N}}
		\dfrac{|v(x)-v(y)|^p}{|x-y|^{N+ps}}dxdy.
		\end{equation}
Combining  \eqref{RR} and \eqref{new03}, we reach that
$$\iint_{\re^{2N}}\dfrac{|v(x)-v(y)|^p}{|x-y|^{N+ps}}dxdy\ge
C\iint_{\re^{2N}}
\dfrac{|u(x)-u(y)|^p}{|x-y|^{N+ps}}\dfrac{dx}{|x|^{\beta}}
\dfrac{dy}{|y|^\beta}.
$$
Hence we conclude.
\end{proof}

\begin{remark}
For $\O\subset \ren$, a bounded regular domain, we define the space $E_{\a,0}(\O)$ as the completion of $\mathcal{C}^\infty_0(\O)$ with respect to the norm
$$
||\phi||_{E_{\a,0}(\O)}=\bigg(\iint_{D_\O}\dfrac{||x|^\a \phi(x)-|y|^\a \phi(y)|^p}{|x-y|^{N+ps}}dxdy\bigg)^{\frac{1}{p}}.
$$
Following the same computations as above we can prove that if $-ps<\beta<\frac{N-ps}{2}$, then $W^{s,p}_{\b,0}(\O)=E_{\a,0}(\O)$ with $\a=-\dfrac{2\b}{p}$.
\end{remark}
\section{Existence Results: $\l\le \L_{N,p,s}$ }\label{sec3}
Recall that we are considering nonnegative solution to problem
\begin{equation}\label{eq:defsec1}
\left\{
\begin{array}{rcll}
u_t+(-\D^s_{p}) u&=&\dyle \l \frac{u^{p-1}}{|x|^{ps}}  &
\text{ in } \O_{T}=\Omega \times (0,T)  , \\
u&\ge& 0 &
\text{ in }\ren\times (0,T), \\
u&=&0 & \text{ in }(\ren\setminus\O) \times (0,T), \\
u(x,0)&=&u_0(x)& \mbox{  in  }\O,
\end{array}%
\right.
\end{equation}
where $\l\le \L_{N,p,s}$ and $u_0\in L^2(\O)$ with $u_0\gneqq 0$. Define $u_{0n}=T_n(u_0)$, starting with $u_{00}\equiv 0$, for $n\ge 1$, we consider $u_n$ as the unique nonnegative solution to the following approximated problem
\begin{equation}\label{pro:lineal1}
\left\{\begin{array}{rcll}
u_{nt}+(-\D^s_{p})u_n &=&\l\dfrac{u^{p-1}_{n-1}}{|x|^{ps}+\frac{1}{n}}& \mbox{  in  }\O_T,\\
u_n&=&0 &\mbox{  in
 } (\ren\backslash\O)\times (0,T),\\
u_n(x,0)&=&u_{0n}(x) & \mbox{ in }\O.
\end{array}
\right.
\end{equation}
The existence of $u_n$ follows using classical arguments for monotone operator as in \cite{Lio}. It is clear that $u_n\gneqq 0$ and $\{u_n\}_n$ is monotone in $n$. As a consequence we get the first existence result.
\begin{Theorem}
Assume that $\l<\L_{N,p,s}$ and $u_0 \in L^2(\O)$, then the problem \eqref{eq:defsec1} has global solution $u$ such that $u\in L^{p}(0,T; W^{s,p}_0(\O))\cap \mathcal{C}([0,T], L^p(\O))$, $u_t\in L^{p'}(0,T; W^{-s,p'}_0(\O))$ and $u(x,.)\to u_0$ strongly in $L^2(\O)$ as $t\to 0$.
\end{Theorem}
\begin{proof}
Using  $u_n$ as a test function in \eqref{pro:lineal1}, we get
\begin{equation*}
\frac{1}{2}\frac{d}{dt}\int_{\Omega }u^{2}_ndx+ \frac{1}{2}\iint_{D_{\O}}\dfrac{|u_n(x,t)-u_n(y,t)|^{p}}{|x-y|^{N+ps}}dy \ dx\le \l\io\dfrac{|u_n|^{p}}{|x|^{ps}}dx.
\end{equation*}
Using the Hardy-Sobolev inequality we obtain
$$\frac{1}{2}\int_{\Omega }u^{2}_n(x,T)dx+ \Big(\frac{1}{2}-\frac{\l}{\L_{N,p,s}}\Big)\int_0^T\iint_{D_{\O}}\dfrac{|u_n(x,t)-u_n(y,t)|^{p}}{|x-y|^{N+ps}}dy \ dx\,dt\le \frac{1}{2}||u_0||^2_2.$$
Hence we get the existence of a measurable function $u$ such that $u_n\uparrow u$ a.e in $\O_T$, $u_n \rightharpoonup u$ weakly in $L^p(0,T;W^{s,p}_0(\O))$ and
$\dfrac{|u_n|^{p}}{|x|^{ps}}\to \dfrac{|u|^{p}}{|x|^{ps}}$ strongly in $L^1(\O_T)$. Now the rest of the proof follows by using classical compactness arguments.
\end{proof}

In the case where $\l=\L_{N,p,s}$ we can use the improved Hardy-Sobolev inequality given in \eqref{impro}, in this case we can prove the next Theorem.
\begin{Theorem}
Assume that $\l=\L_{N,p,s}$ and $u_0 \in L^2(\O)$, then the problem \eqref{eq:defsec1} has a global solution $u$ such that
$$
\frac{1}{2}\int_0^T\iint_{D_{\O}}\dfrac{|u(x,t)-u(y,t)|^{p}}{|x-y|^{N+ps}}dx \ dy dt-\L_{N,p,s}\iint_{\O_T}\dfrac{|u|^{p}}{|x|^{ps}}dxdt\le \frac{1}{2}||u_0||^2_2.
$$
Moreover, $u\in L^{q}(0,T; W^{s,q}_0(\O))\cap \mathcal{C}([0,T], L^2(\O))$ for all $q<p$.
\end{Theorem}

\section{Existence Results: $p<2$ and $\l> \L_{N,p,s}$ }\label{singular}
Let consider now the more interesting case, $\l>\L_{N,p,s}$ and $p<2$. According to the value of $p$, we will prove that problem \eqref{eq:defsec1}  has a solution.
\subsection{The case $1<p< \frac{2N}{N+2s}$ and $\l>\L_{N,p,s}$}
The main result of this subsection is the following one.
\begin{Theorem}\label{th0}
Assume that $1<p< \frac{2N}{N+2s}$ and $\l>\L_{N,p,s}$. Let $u_0\in L^2(\O)$, then problem \eqref{eq:defsec1} has global solution $u\in L^{p}(0,T; W^{s,p}_0(\O))$.
\end{Theorem}
\begin{proof}
Setting $ W_n(x)=\dfrac{1}{|x|^{ps}+\frac{1}{n}}$, then using a suitable iteration arguments we can prove that the problem
\begin{equation}\label{aprr}
\left\{
\begin{array}{rcll}
u_{t}+(-\D^s_{p}) u&=& \l W_n(x) u^{p-1}  & \text{ in } \O_{T},\\
u(x,t)&=& 0 & \text{ in }(\ren\setminus\O) \times (0,T), \\
u(x,0)&=&u_{0n}(x) & \mbox{  in  }\O,
\end{array}%
\right.
\end{equation}
has a bounded minimal nonnegative solution $u_n$. Using $u_n$ as a test function in \eqref{aprr} and by H\"older and Young inequalities, it follows that
\begin{eqnarray*}
 & \dint_{\O}u^2_n(x,T)\ dx + \dint_{0}^T \iint_{D_\O}\dfrac{|u_n(x,t)-u_n(y,t)|^{p}}{|x-y|^{N+ps}}\,dx dy\ dt\\ &=
\dint_{\O}u^2_{0n}(x)\ dx + \l \iint_{\O_T} W_n(x) u^p_n(x,t)\,dx dt\\
& \le \dint_{\O}u^2_{0}(x)\ dx + \l \int_{0}^T \Big(\dint_{\O} |W_n(x)|^{2/(2-p)}\,dx\Big)^{(2-p)/2} \times \Big(\dint_{\O}u^2_n(x,t)\,dx\Big)^{p/2} dt \\
& \le \dint_{\O}u^2_{0}(x)\ dx + \l \Big( \frac{2-p}{2} \iint_{\O_T}|W_n(x)|^{2/(2-p)}\,dx dt+\frac{p}{2} \iint_{\O_T}u^2_n(x,t)\,dx dt\Big).
\end{eqnarray*}
Calling $$y_n(T)\equiv \dint_{\O}  |u_n(x,T)|^{2}\,dx$$
and
\begin{eqnarray*}
\b_n(T) & = & \dint_{\O}|u_{0}(x)|^{2}\ dx + \l  \frac{2-p}{2} \iint_{\O_T}|W_n(x)|^{2/(2-p)}\,dx dt\\
&\le & C\bigg(\dint_{\O}|u_{0}(x)|^{2}\ dx+ \iint_{\O_T}\Big(\frac{1}{|x|^{\frac{2ps}{2-p}}}\Big)^{2/(2-p)}\,dx\,dt\bigg).
\end{eqnarray*}
Since $1<p<2N/(N+2s)$, then $\b_n(T)\le C(T+1)$. Thus
$$y_n(T) \leq \b_n(t) + \l \frac{p}{2} \dint_{0}^T y_n(s) ds,$$
and, as a consequence of Gronwall inequality, we reach that
\begin{equation}
\dint_{\O}  |u_n(x,T)|^{2}\,dx \leq \b_n(T) + \dint_{0}^T \b_n(s) e^{\a s} ds,
\end{equation}
where $\a=\a(\l,p)>0$. Thus
$$
\dint_{\O}u^2_n(x,T)\ dx + \dint_{0}^T \iint_{D_\O}\dfrac{|u_n(x,t)-u_n(y,t)|^{p}}{|x-y|^{N+ps}}\,dx dy\ dt\le C(T)
$$
and
$$
\iint_{\O_T} W_n(x) u^p_n(x,t)\,dx dt\le C(T).
$$
Hence we get the existence of a measurable function $u$ such that $u_n\uparrow u$ a.e in $\O_T$, $u_n\rightharpoonup u$ weakly in $L^{p}(0,T; W^{s,p}_0(\O))$ and $u_t\in L^{p'}(0,T; W^{-s,p'}_0(\O))$. It is not difficult to show that $u$ is globally defined in the time and that $u$ solves problem \eqref{eq:defsec1}.
\end{proof}
\subsection{The case $\l >\l_{N,p,s}$ and  $2 > p \geq 2N/(N+2s)$}

We begin by investigating the existence of a nonnegative selfsimilar solution for the Cauchy problem in the whole $\ren$.

We set $V(x,t)=t^{\a}F(x)$, then
\begin{equation}\label{time}
V_t=\a t^{\a-1}F(r) \mbox{  and  } (-\D^s_{p})V(x,t)=t^{\a(p-1)}\dyle\int_{\ren} \dfrac{|F(x)-F(y)|^{p-2}(F(x)-F(y))}{|x-y|^{N+ps}}dy.
\end{equation}
Setting $\a=\frac{1}{p-2}$, it holds
\begin{eqnarray}\label{good}
& \a F(x) +\dyle \int_{\ren} \,\dfrac{|F(x)-F(y)|^{p-2}(F(x)-F(y))}{|x-y|^{N+ps}}dy=\l\frac{|F(x)|^{p-2}F(x)}{|x|^{ps}}.
\end{eqnarray}
Let us search $F$ in the form $F(x)= F(|x|)=A|x|^\g$,
then \eqref{good} implies that
\begin{equation}\label{hh}
\a A r^\g + A^{p-1}r^{\g(p-1)-ps}\int_0^\infty\,|1-\sigma^\g|^{p-2}(1-\sigma^\g)\sigma^{N-1}K(\sigma) d\sigma=\l A^{p-1}r^{\g(p-1)-ps}
\end{equation}
where, as in \eqref{kkk},
$$
K(\s)=\dint\limits_{|y'|=1}\dfrac{dH^{n-1}(y')}{|x'-\s
y'|^{N+ps}}.
$$
Assume that $\g= \frac{-ps}{2-p}$, then $\gamma<0$ and $\g=\g(p-1)-ps$. Hence by \eqref{hh} we obtain that
$$A^{p-2}=\dfrac{\a}{\Psi(\g)+\l }\equiv B,$$
where $$\Psi(\g)=\int_0^\infty\,|\sigma^\g-1|^{p-2}(\sigma^\g-1)K(\sigma)\sigma^{N-1} d\sigma.$$
Let $\bar{\gamma}=-\gamma$, then $\bar{\gamma}>0$ and
$$\Psi(\gamma)=\int_0^\infty\,|1-\sigma^{\bar{\gamma}}|^{p-2}(1-\sigma^{\bar{\gamma}})K(\sigma)\sigma^{-\bar{\gamma}(p-1)+N-1} d\sigma\equiv \Psi_1(\bar{\gamma}).$$
Therefore,
$$B=\dfrac{1}{(2-p)(\Psi_1(\bar{\gamma})+\l) }$$ and then
$$
V(x,t)=B^{\frac{1}{p-2}}\Big(\dfrac{t}{|x|^{ps}}\Big)^{\frac{1}{2-p}}.
$$
Notice that, in the local, for $\l>\L_{N,p,1}$, the positivity of $B$ follows using a simple algebraic inequality. The situation is more complicated in the nonlocal case and some fine computations are needed.

Since
\begin{eqnarray}\label{TT}
\Psi_1(\bar{\gamma}) & = & \int_0^\infty\,|1-\sigma^{\bar{\gamma}}|^{p-2}(1-\sigma^{\bar{\gamma}})K(\sigma)\sigma^{-\bar{\gamma}(p-1)+N-1} d\sigma=\dyle \int_0^1  +  \int_1^\infty\\
&= & I_1+I_2,
\end{eqnarray}
then taking into consideration that $K(\frac{1}{\xi})=\xi^{N+ps}K(\xi)$ for $\xi>0$ and using the change of variable $\xi=\frac{1}{\sigma}$ in $I_2$,
there results that
\begin{equation}\label{psi11}
\Psi_1(\bar{\g})=\dint\limits_1^{\infty}K(\sigma)(\sigma^{\bar{\gamma}}-1)^{p-1}\left(\sigma^{ps-1}-\sigma^{N-1-\bar{\g}(p-1)}\right)\,d\sigma.
\end{equation}

Let us begin by proving the next Lemma.
\begin{lemma}\label{intt}
Assume that $1<p<2$ and $\l>\L_{N,p,s}$ then $B>0$
\end{lemma}

To prove Lemma \ref{intt} we need the next result( see Lemma 3.1 in \cite{AB}).
\begin{Proposition}\label{prop1}
Let
$$
\Theta(\eta)=\dint\limits_1^{+\infty}K(\sigma)(\sigma^\eta-1)^{p-1}\left(\sigma^{N-1-\eta(p-1)}-\sigma^{ps-1}\right)\,d\sigma,
$$
where $\eta\ge 0$ and $\eta<\frac{N-ps}{p-1}$, then we have
\begin{enumerate}
\item $\Theta(0)=0$ and  $\Theta(\frac{N-ps}{p})=\L_{N,p,s}=\max\limits_{\eta\ge 0}\Theta(\eta)$.
\item For all $0<\l<\L_{N,p,s}$, then there exist $\rho_1,\rho_2$ such
that $0<\eta_1<\frac{N-ps}{p}<\eta_2$ and $\Theta(\eta_1)=\Theta(\eta_2)=\l$.
\end{enumerate}
\end{Proposition}

{\bf Proof of Lemma \ref{intt}.}

We have just to show that $(\Psi_1(\bar{\gamma})+\l)>0$. We split our work in two cases according to the value of $p$.

{\em The first case: $\frac{2N}{N+s}\le p<2$.}

In this case we have $ps\ge N-\bar{\gamma}(p-1)$, thus using \eqref{psi11} we get easily that $\Psi_1(\bar{\g})\ge 0$. Hence $B>0$ and the result follows in this case.

{\em The second case: $p<\frac{2N}{N+s}$.}

This the more delicate case. It is clear that $\bar{\gamma}<\frac{N-ps}{p-1}$ and that
\begin{equation}\label{psi110}
\Psi_1(\bar{\g})=-\dint\limits_1^{\infty}K(\sigma)(\sigma^{\bar{\gamma}}-1)^{p-1}\left(\sigma^{N-1-\bar{\g}(p-1)}-\sigma^{ps-1}\right)\,d\sigma<0.
\end{equation}
By \eqref{psi110} we have $\Psi_1(\bar{\g})=-\Theta(\bar{\g})\ge -\L_{N,p,s}$. Since $\l>\L_{N,p,s}$  by Proposition \ref{prop1}, we reach that
$\l+\Psi_1(\bar{\g})>\l-\L_{N,p,s}>0$. Hence we conclude. \cqd

\

It's easy to see that self-Similar solution obtained above is a super solution in the distribution sense to problem \eqref{eq:defsec1} if and only if  $p<\frac{2N}{N+s}$. It is clear that $\frac{V^{p-1}}{|x|^{ps}}\in L^1(\O_T)$. Hence in this case, using monotony argument we get the next existence result.
\begin{Theorem}\label{sing1}
Assume that $\frac{2N}{N+2s}<p<\frac{N}{N+s}$ and $\l>\L_{N,p,s}$. Let $u_0$ be such that $u_0(x)\le \frac{C}{|x|^{\frac{ps}{2-p}}}$, then problem \eqref{eq:defsec1} has global entropy solution $u$ such that $u\le V$ and
\begin{equation}\label{eq:eq1}
\int_0^T\iint_{\O\times \O}\dfrac{|u(x,t)-u(y,t)|^{q}}{|x-y|^{N+qs}}d\nu\ dt\le M
\end{equation}
for all $q<p_2=\frac{N(p-1)+ps}{N+s}$.
\end{Theorem}

In the case where $p\ge \frac{2N}{N+s}$, then $V\notin L^1(B_\d(0)\times (0,T))$ for any $\d>0$. However we can show that in this case we have a solution away from the origin and that solution is in a suitable fractional Sobolev space with degenerate weight. This is the main goal of the next computations.

Let us define the next weighted parabolic Sobolev space.
$$\Upsilon_\a=\bigg\{u: |x|^\a u\in L^p(0,T; W^{s,p}_0(\O))\cap C^0([0,T];L^2(\O)), |x|^{\a(p-1)} u_t\in L^{p'}(0,T; W^{-s,p'}_0(\O)) \bigg\}.$$
Notice that if $u\in \Upsilon_\a$, then $u\in L^p(0,T, E_0(\O))$, we refer to the Subsection \ref{degenerate} for some useful properties of the space $E_{\a, 0}(\O)$.

Then we have the next theorem.
\begin{Theorem}
Assume that $\l >\L_{N,p,s}$ and $\frac{2N}{N+s}\le p<2$. Suppose that $u_0\in L^2(|x|^\a dx, \O)$ for some $\a>\frac{2s}{2-p}-\frac{N}{p}$. Then there exists a function $u\in \Upsilon_\a$ which is a solution to \eqref{eq:defsec1}  away from the origin. Moreover for all $v\in \Upsilon_\a$, we have
\begin{equation}\label{distt}
\begin{array}{lll}
&-\dyle\int_0^T\langle v_t, |x|^{p\a} u\rangle dt +\dfrac 12\int_0^T\iint_{D_{\O}}U(x,y,t)(|x|^{p\a} v(x,t)-|y|^{p\a} v(y,t))d\nu\ dt\\
&\dyle=\l\iint_{\O_T} \frac{u^{p-1}|x|^{p\a} v}{|x|^{ps}} dx\,dt.
\end{array}
\end{equation}

\end{Theorem}
\begin{proof} Recall that $u_n$ is the unique solution to the approximated problem \eqref{pro:lineal1}.
Let $w(x)=|x|^{p\a}$ where $\a>\frac{2s}{p-2}-\frac{N}{p}$ and define
$$
U_n(x,y,t)=|u_n(x,t)-u_n(y,t)|^{p-2}(u_n(x,t)-u_n(y,t)).
$$
Using $wu_n$ as a test function in \eqref{pro:lineal1}, it follows that
$$\int_{\O}u_{nt}wu_ndx+\int_{\O}(-\D^s_{p})u_n\,w u_ndx=\l\int_{\O}\dfrac{u^{p-1}_{n-1}u_n}{|x|^{ps}+\frac{1}{n}}w dx.$$
Integrating in the time and using the fact that the sequence $\{u_n\}_n$ is increasing in $n$, it follows that
\begin{eqnarray*}
&\dyle\frac{1}{2}\int_{\O}u^2_{n}(x,T)w(x)dx +\int_0^T\iint_{D_{\O}}U_n(x,y,t)(u_n(x,t)w(x)-u_n(y,t)w(y))d\nu\ dt\\
&\le \dyle \l\iint_{\O_T}\dfrac{(u_n(x,t)w^{\frac 1p}(x))^{p}}{|x|^{ps}}\,dx\,dt+\frac{1}{2}\int_{\O}u^2_{0}(x)w(x)dx.
\end{eqnarray*}
Using inequality \eqref{alge4}, we obtain
\begin{eqnarray*}
& U_n(x,y,t)(u_n(x,t)w(x)-u_n(y,t)w(y))\ge\\ & C_1|u_n(x,t)w(x)^{\frac{1}{p}}-u_n(y,t)w(y)^{\frac{1}{p}}|^{p}-C_2(u^p_n(x,t)+u^p_n(y,t))|w(x)^{\frac{1}{p}}-w(y)^{\frac{1}{p}}|^{p}.
\end{eqnarray*}
Hence we conclude that
\begin{equation}\label{last000}
\begin{array}{lll}
&\dyle\frac{1}{2}\int_{\O}u^2_{n}(x,T)w(x)\,dx+C_1\int_0^T\iint_{D_{\O}}|u_n(x,t)w(x)^{\frac{1}{p}}-u_n(y,t)w(y)^{\frac{1}{p}}|^{p}d\nu\ dt\\
&\le \dyle C_2 \int_0^T\iint_{D_{\O}}(u^p_n(x,t)+u^p_n(y,t))|w(x)^{\frac{1}{p}}-w(y)^{\frac{1}{p}}|^{p}d\nu\ dt\\
& \dyle +\l\iint_{\O_T}\dfrac{(u_nw^{\frac 1p})^{p}}{|x|^{ps}}\,dx\,dt+\frac{1}{2}\int_{\O}u^2_{0}(x)w(x)dx.
\end{array}
\end{equation}
Let us analyze the term $\dyle\int_0^T\iint_{D_{\O}}(u^p_n(x,t)+u^p_n(y,t))|w(x)^{\frac{1}{p}}-w(y)^{\frac{1}{p}}|^{p}d\nu dt$. Using symmetric arguments we get
$$
\begin{array}{lll}
& \dyle \int_0^T\iint_{D_{\O}}(u^p_n(x,t)+u^p_n(y,t))|w(x)^{\frac{1}{p}}-w(y)^{\frac{1}{p}}|^{p}d\nu dt\\
&=\dyle 2\int_0^T\iint_{D_{\O}}u^p_n(x,t)|w(x)^{\frac{1}{p}}-w(y)^{\frac{1}{p}}|^{p}d\nu dt\equiv 2J.
\end{array}
$$
We claim that
\begin{equation}\label{dege0}
J\le C\iint_{\O_T}\dfrac{u^p_nw}{|x|^{ps}}dxdt.
\end{equation}
Since $\Omega$ is a bounded domain, then $\Omega\subset\subset B_R(0)$, hence
$$
J \le  \int_0^T\int_{B_R(0)}u^p_{n}(x,t)\irn ||x|^{\a}-|y|^{\a}|^{p}d\nu dt.
$$
We set $r=|x|$ and $\rho=|y|$, then $x=rx', y=\rho y'$.
where $|x'|=|y'|=1$. Therefore we obtain that

\begin{eqnarray*}
J &\le & \int_0^T\int_{B_R(0)}u^p_{n}(x,t)|x|^{-ps}\int_0^{\infty}|r^{\a}-(r\s)^{\a}|^p\sigma^{N-1}
\left(\dint\limits_{|y'|=1}\dfrac{dH^{n-1}(y')}{|x'-\s
y'|^{N+ps}} \right)d\s\,dx\,dt\\
&\le & \int_0^T\int_{B_R(0)}u^p_{n}(x,t)|x|^{p\a-ps}\int_0^{\infty}|1-\s^{\a}|^p\sigma^{N-1}
K(\s)d\s \,dx\,dt\\
\end{eqnarray*}
Setting $C=\dyle\int_0^{\infty}|1-\s^{\a}|^p\sigma^{N-1} K(\s)d\s$, taking into consideration the behavior of $K$ near to $1$ and $\infty$, we can prove that $C<\infty$. Hence, since $u_n=0$ in $(\ren\setminus\O) \times (0,T)$, we get
$$
J\le C\iint_{\O_T}\dfrac{u^p_nw}{|x|^{ps}}\,dx\,dt
$$
and the claim follows.

Now, as $p<2$, using Young inequality,
$$
\iint_{\O_T}\dfrac{u^p_nw}{|x|^{ps}}\,dx\,dt\le C_3 \iint_{\O_T}u^2_nw(x)dx\,dt +C_4\iint_{\O_T}|x|^{p\a-\frac{2ps}{2-p}}\,dx\,dt.
$$
Since $\a>\frac{2s}{2-p}-\frac{N}{p}$, it holds that $\dyle\iint_{\O_T}|x|^{p\a-\frac{2ps}{2-p}}\,dx\,dt\le C_5 T$. Thus
\begin{equation}\label{dege2}
\iint_{\O_T}\dfrac{u^p_nw}{|x|^{ps}}\,dx\,dt \le C_3 \iint_{\O_T}u^2_n(x,t)w(x)\,dx\,dt +C T.
\end{equation}
Going back to \eqref{last000}, by \eqref{dege0} and \eqref{dege2}, we reach that
\begin{eqnarray*}
&\dyle\frac{1}{2}\int_{\O}u^2_{n}(x,T)w(x)dx +C_1\int_0^T\iint_{D_{\O}}|u_n(x,t)w(x)^{\frac{1}{p}}-u_n(y,t)w(y)^{\frac{1}{p}}|^{p}d\nu\ dt\\
&\le \dyle C_2 \iint_{\O_T}u^2_n(x,t)w(x)\,dx\,dt +C_3T+C_4.
\end{eqnarray*}
Using Gronwall Lemma we obtain that $\int_{\O}u^2_{n}(x,T)w(x)dx\le C(T)$ and then
$$
\int_0^T\iint_{D_{\O}}|u_n(x,t)w(x)^{\frac{1}{p}}-u_n(y,t)w(y)^{\frac{1}{p}}|^{p}d\nu\ dt\le C(T).
$$
We set $\tilde{u}_n=w(x) u_n$, then $\{\tilde{u}_n\}_n$ is increasing in $n$ and bounded in the space $L^p(0,T; W^{s,p}_0(\O))$. Hence we get the existence of a measurable function $u$ such that $u_n\uparrow u$ a.e. in $\O$, $u=0$ in $(\ren\setminus\O) \times (0,T)$ and $\tilde{u}_n\rightharpoonup w(x)u$ weakly in $L^p(0,T; W^{s,p}_0(\O))$. Let $U(x,y,t)=|u(x,t)-u(y,t)|^{p-2}(u(x,t)-u(y,t))$, then $U_n\to U$ a.e. in $D_{\O}\times (0,T)$.

Let us show that $u$ satisfies \eqref{distt}. Let $v\in \mathcal{C}^\infty_0(\O_T)$, using $wv$ as a test function in the approximating problem \eqref{pro:lineal1} and integrating in the time, it follows that
\begin{eqnarray*}
&-\dyle\int_0^T\int_{\O}v_t u_{n}(x,t)w(x)\,dx\,dt+\int_{\O}v(x,T)u_n(x,T)w(x)dx\\
&+\dyle \int_0^T\iint_{D_{\O}}U_n(x,y,t)(v(x,t)w(x)-v(y,t)w(y))d\nu\ dt\\
&=\dyle \l\iint_{\O_T}\dfrac{u^{p-1}_{n-1}wv}{|x|^{ps}+\frac 1n}\,dx\,dt+\int_{\O}u^2_{n0}(x)v(x,0)w(x)dx.
\end{eqnarray*}
Taking into consideration the previous estimates, we get easily that, as $n\to \infty$,
\begin{eqnarray*}
&-\dyle\iint_{\O_T}v_t u_{n}(x,t)w(x)\,dx\,dt+\int_{\O}v(x,T)u_n(x,T)w(x)dx\to \\
& -\dyle\iint_{\O_T}v_t u(x,t)w(x)\,dx\,dt+\int_{\O}v(x,T)u(x,T)w(x)dx
\end{eqnarray*}
and
\begin{eqnarray*}
&\dyle \iint_{\O_T}\dfrac{u^{p-1}_{n-1}wv}{|x|^{ps}+\frac 1n}\,dx\,dt+\int_{\O}u^2_{n0}(x)v(x,0)w(x)dx\to \\
& \dyle \iint_{\O_T}\dfrac{u^{p-1}wv}{|x|^{ps}}\,dx\,dt+\int_{\O}u^2_{0}(x)v(x,0)w(x)dx.
\end{eqnarray*}
Let us prove that
\begin{eqnarray*}
&\dyle \int_0^T\iint_{D_{\O}}U_n(x,y,t)(v(x,t)w(x)-v(y,t)w(y))d\nu\ dt\to \\
&\dyle \int_0^T\iint_{D_{\O}}U(x,y,t)(v(x,t)w(x)-v(y,t)w(y))d\nu\ dt.
\end{eqnarray*}
Define
$$\tilde{u}_n=|x|^\a u_n,\,\, \tilde{u}=|x|^\a u,\,\,\, \tilde{v}=|x|^\a v,$$
$$
\tilde{U}_n(x,y,t)=|\tilde{u}_n(x,t)-\tilde{u}_n(y,t)|^{p-2}(\tilde{u}_n(x,t)-\tilde{u}_n(y,t)), $$ $$\tilde{U}(x,y,t)=|\tilde{u}(x,t)-\tilde{u}(y,t)|^{p-2}(\tilde{u}(x,t)-\tilde{u}(y,t))
$$
and
$$
\tilde{V}(x,y,t)=|\tilde{v}(x,t)-\tilde{v}(y,t)|^{p-2}(\tilde{v}(x,t)-\tilde{v}(y,t)).
$$
Using the previous estimates on $\{u_n\}_n$, we have $\tilde{u}_n, \tilde{u}, \tilde{v}\in L^p(0,T; W^{s,p}_0(\O))$, $\{\tilde{u}_n\}_n$ is bounded in $L^p(0,T; W^{s,p}_0(\O))$ and $\tilde{u}_n\rightharpoonup \tilde{u}$ weakly in $L^p(0,T; W^{s,p}_0(\O))$.

We have
$$
U_n(x,y,t)(v(x,t)w(x)-v(y,t)w(y))=J_n(x,y,t)+L_n(x,y,t),
$$
where
$$
J_n(x,y,t)=\tilde{U}_n(x,y,t)(\tilde{v}(x,t)-\tilde{v}(y,t))
$$
and \begin{eqnarray*}
& L_n(x,y,t)=\\
&\bigg|(\tilde{u}_n(x,t)-\tilde{u}_n(y,t))+(1-(\frac{|x|}{|y|})^\a)\tilde{u}_n(y,t)\bigg|^{p-2}
\bigg((\tilde{u}_n(x,t)-\tilde{u}_n(y,t))+(1-(\frac{|x|}{|y|})^\a)\tilde{u}_n(y,t)\bigg)\\
& \times \bigg((\tilde{v}_n(x,t)-\tilde{v}_n(y,t))+(1-(\frac{|x|}{|y|})^{\a(p-1)})\tilde{v}_n(y,t)\bigg)-J_n(x,y,t).
\end{eqnarray*}
Using a duality argument we reach that
\begin{eqnarray*}
&\dyle \int_0^T\iint_{D_{\O}}J_n(x,y,t)d\nu\ dt\to \int_0^T\iint_{D_{\O}}J(x,y,t)d\nu\ dt\mbox{  as  }n\to \infty.
\end{eqnarray*}
We deal now with $L_n$. It is clear that $L_n\to L_n$ a.e in $D_\O\times (0,T)$, where
\begin{eqnarray*}
&L(x,y,t)=\\
& \bigg|(\tilde{u}(x,t)-\tilde{u}(y,t))+(1-(\frac{|x|}{|y|})^\a)\tilde{u}(y,t)\bigg|^{p-2}
\bigg((\tilde{u}(x,t)-\tilde{u}(y,t))+(1-(\frac{|x|}{|y|})^\a)\tilde{u}(y,t)\bigg)\\
& \times \bigg((\tilde{v}(x,t)-\tilde{v}(y,t))+(1-(\frac{|x|}{|y|})^{\a(p-1)})\tilde{v}(y,t)\bigg)-J(x,y,t).
\end{eqnarray*}
It is clear that
$$
|L_n(x,y,t)|\le L_{n1}(x,y,t)+L_{n2}(x,y,t),
$$
where
\begin{eqnarray*}
& L_{n1}(x,y,t)=\\
& \Bigg\|\bigg|(\tilde{u}_n(x,t)-\tilde{u}_n(y,t))+(1-(\frac{|x|}{|y|})^\a)\tilde{u}_n(y,t)\bigg|^{p-2}
\bigg((\tilde{u}_n(x,t)-\tilde{u}_n(y,t))+(1-(\frac{|x|}{|y|})^\a)\tilde{u}_n(y,t)\bigg)\\
&-\bigg|\tilde{u}_n(x,t)-\tilde{u}_n(y,t)\bigg|^{p-2}
\bigg(\tilde{u}_n(x,t)-\tilde{u}_n(y,t)\bigg)\Bigg\|\\
& \times\bigg|\tilde{v}(x,t)-\tilde{v}(y,t)\bigg|,
\end{eqnarray*}
and
\begin{eqnarray*}
L_{n2}(x,y,t) &=&\bigg|(\tilde{u}_n(x,t)-\tilde{u}_n(y,t))+(1-(\frac{|x|}{|y|})^\a)\tilde{u}_n(y,t)\bigg|^{p-1}\\
& \times & \bigg|(1-(\frac{|x|}{|y|})^{\a(p-1)})\tilde{v}(y,t)\bigg|.
\end{eqnarray*}
Hence
\begin{eqnarray*}
L_{n2}(x,y,t) &\le &\bigg|\tilde{u}_n(x,t)-\tilde{u}_n(y,t)\bigg|^{p-1}\times\bigg|(1-(\frac{|x|}{|y|})^{\a(p-1)})\tilde{v}(y,t)\bigg|\\
&+& \bigg|(1-(\frac{|x|}{|y|})^\a)\tilde{u}_n(y,t)\bigg|^{p-1}\times \bigg|(1-(\frac{|x|}{|y|})^{\a(p-1)})\tilde{v}(y,t)\bigg|\\
&\le & L_{n21}(x,y,t)+L_{n22}(x,y,t).
\end{eqnarray*}
We claim that $\bigg|(1-(\frac{|x|}{|y|})^{\a(p-1)})\tilde{v}(y,t)\bigg|\in L^p(D_\O\times (0,T),d\nu\ dt)$.
We have
\begin{eqnarray*}
&\dyle \int_0^T\iint_{D_{\O}}\bigg|(1-(\frac{|x|}{|y|})^{\a(p-1)})\tilde{v}(y,t)\bigg|^pd\nu\ dt\le\\
&\dyle \int_0^T\io \dfrac{|\tilde{v}(y,t)|^p}{|y|^{p\a(p-1)+ps}}\int_{\ren}\frac{\bigg||y|^{\a(p-1)}-|x|^{\a(p-1)}\bigg|^p}{|x-y|^{N+ps}}dx\,dy\,dt.
\end{eqnarray*}
We set $r=|x|, \rho=|y|$, then $x=rx', y=\rho y'$
where $|x'|=|y'|=1$. For $\sigma=\dfrac{r}{\rho}$, it holds that
\begin{eqnarray*}
&\dyle \int_0^T\iint_{D_{\O}}\bigg|(1-(\frac{|x|}{|y|})^{\a(p-1)})\tilde{v}(y,t)\bigg|^pd\nu\ dt\\
&\le \dyle \iint_{\O_T}\dfrac{|\tilde{v}(y,t)|^p}{|y|^{p\a(p-1)+ps}}\int_0^\infty \bigg||y|^{\a(p-1)}-(\s|y|)^{\a(p-1)}\bigg|^p \s^{N-1}K(\s) d\s\, dy\, dt \\
&\le \dyle \iint_{\O_T}\dfrac{|\tilde{v}(y,t)|^p}{|y|^{ps}}\int_0^\infty \bigg|1-\s^{\a(p-1)}\bigg|^p \s^{N-1}K(\s)\, d\s\, dy\,dt \\
&\le \dyle C\iint_{\O_T}\dfrac{|\tilde{v}(y,t)|^p}{|y|^{ps}}\,dy\, dt,
\end{eqnarray*}
where
$$
C=\int_0^\infty \bigg|1-\s^{\a(p-1)}\bigg|^p \s^{N-1}K(\s) d\s<\infty.
$$
Since $\tilde{v}\in L^p(0,T; W^{s,p}_0(\O))$, using the Hardy inequality, it follows that $\dyle\iint_{\O_T}\dfrac{|\tilde{v}(y,t)|^p}{|y|^{ps}}\,dy\, dt<\infty$ and then the claim follows.

Therefore $L_{n21}$ converges strongly in $L^1(D_\O\times (0,T),d\nu\ dt)$. In the same way we can prove that $L_{n22}$ converge strongly in $L^1(D_\O\times (0,T),d\nu\ dt)$. Hence using the Dominated convergence theorem we obtain that $L_{n2}$ converges to $L_2$ strongly in $L^1(D_\O\times (0,T),d\nu\ dt)$ where
\begin{eqnarray*}
L_{2}(x,y,t) &=&\bigg|(\tilde{u}(x,t)-\tilde{u}(y,t))+(1-(\frac{|x|}{|y|})^\a)\tilde{u}(y,t)\bigg|^{p-1}\\
& \times & \bigg|(1-(\frac{|x|}{|y|})^{\a(p-1)})\tilde{v}(y,t)\bigg|.
\end{eqnarray*}

Now, since $p<2$, then
\begin{eqnarray*}
L_{n1}(x,y,t) &\le & C\bigg|(1-(\frac{|x|}{|y|})^\a)\tilde{u}_n(y,t)\bigg|^{p-1}
\times \bigg|(\tilde{v}(x,t)-\tilde{v}(y,t))\bigg|,
\end{eqnarray*}
Since $\tilde{v}\in L^p(0,T; W^{s,p}_0(\O))$, then using the same computations as in the previous claim, we reach that $\bigg|(1-(\frac{|x|}{|y|})^\a)\tilde{u}_n(y,t)\bigg|^{p-1}\in L^{\frac{p}{p-1}}(D_\O\times (0,T),d\nu\ dt)$. Therefore using the Dominated convergence theorem it follows that  $L_{n1}\to L_1$ converges to $L_1(x,y,t)$ strongly in $L^1(D_\O\times (0,T),d\nu\ dt)$ where
\begin{eqnarray*}
& L_{1}(x,y,t)=\\
&\Bigg\|\bigg|(\tilde{u}(x,t)-\tilde{u}(y,t))+(1-(\frac{|x|}{|y|})^\a)\tilde{u}_n(y,t)\bigg|^{p-2}
\bigg((\tilde{u}(x,t)-\tilde{u}(y,t))+(1-(\frac{|x|}{|y|})^\a)\tilde{u}_n(y,t)\bigg)\\
&- \bigg|\tilde{u}(x,t)-\tilde{u}(y,t)\bigg|^{p-2}
\bigg(\tilde{u}(x,t)-\tilde{u}(y,t)\bigg)\Bigg\|\\
& \times \bigg|(\tilde{v}(x,t)-\tilde{v}(y,t))\bigg|.
\end{eqnarray*}
Combining the above estimates, we conclude that
\begin{eqnarray*}
&\dyle \int_0^T\iint_{D_{\O}}U_n(x,y,t)(v(x,t)w(x)-v(y,t)w(y))d\nu\ dt\to \\
&\dyle \int_0^T\iint_{D_{\O}}U(x,y,t)(v(x,t)w(x)-v(y,t)w(y))d\nu\ dt.
\end{eqnarray*}
Hence $u\in \Upsilon_\a$ satisfies \eqref{distt}.  It is clear that $u$ is a distributional solution to \eqref{eq:defsec1} in $\O\backslash \{0\}\times (0,T)$.
\end{proof}

\section{The singular case $p<2$: Further properties of the solutions. }\label{singg}

In this section we suppose that $p<2$, our main goal is to get natural condition on the data in order to show the existence or non existence of finite time extinction. The first result in this direction is the following.

\begin{Theorem}\label{exis2}
Assume that $\l<\L_{N,p,s}$ and define $u$ to be the minimal solution to the problem
\begin{equation}\label{propa}
\left\{
\begin{array}{rcll}
u_t+(-\D^s_{p}) u & = & \l\dfrac{u^{p-1}}{|x|^{ps}}  & \text{ in } \O_{T}, \\
u & \ge &  0 & \text{ in }\ren\times (0,T), \\
u & = & 0 & \text{ in }(\ren\setminus\O) \times (0,T), \\
u(x,0) & = & u_0(x) & \mbox{  in  }\O,
\end{array}%
\right.
\end{equation}
then we have
\begin{enumerate}
\item if $\frac{2N}{N+2s}\le p<2$ and $u_0\in L^2(\O)$, there exists a finite time $T^*(N,p,|\O|,\L_{N,p,s},||u_0||_2)\equiv T^*\ge ||u_0||^{2-p}_2
    |\O|^{\frac{p}{2}-1+\frac{ps}{N}}$ such that $u(.,t)\equiv 0$ for $t\ge T^*$.
\item if $1<p<\frac{2N}{N+2s}$ and $u_0 \in L^{\nu+1}(\O)\cap L^{2}(\O)$ with $\nu+1=\frac{N(2-p)}{ps}$, there exists $C(N,p,s)>0$ such that
if $\l<C(N,p,s)$, then $u( .,t) \equiv 0$ for all $t\geq T^*$ where $T^*=T^*(\l,C, u_0)$.
\end{enumerate}
\end{Theorem}

\begin{proof}
We follow closely the arguments used in \cite{AP}. Using $u$ as a test function in \eqref{propa}, we get
\begin{equation*}
\frac{1}{2}\frac{d}{dt}\int_{\Omega }u^{2}dx+ \frac{1}{2}\iint_{D_{\O}}\dfrac{|u(x,t)-u(y,t)|^{p}}{|x-y|^{N+ps}}dx \ dy=  \l\int_{\Omega }\dfrac{|u|^{p}}{|x|^{ps}}dx.
\end{equation*}
By Sobolev and Hardy inequalities, we reach that
\begin{equation*}
\frac{1}{2}\frac{d}{dt}\int_{\Omega }u^{2}dx+\frac{C(S,\L_{N,p,s})}{2} \left( \int_{\Omega }\left\vert u\right\vert ^{p^{\ast }_s}dx\right) ^{\frac{p}{p^\ast_s }}\leq 0.
\end{equation*}
Since $\frac{2N}{N+2s}<p<2$, then $p^*_s>2$, thus by H\"older inequality, we obtain
$$\int_{\O}u^2(x,t)dx\le C(\O)\bigg(\int_{\O}\mid u^{p^*_s}(x,t)|\ dx\bigg)^{\frac{2}{p^*_s}}.$$
Thus
$$\frac{1}{2}\frac{d}{dt}\parallel u(x,t)\parallel^2_2+c(\L_{N,p,s}) |\O|^{\frac{p}{p^*_s}-\frac{p}{2}}\parallel u(x,t)\parallel^p_2\le 0.
$$
As a conclusion we reach that
$$\parallel u(x,T)\parallel_2\le \parallel u_0\parallel_2\Big(1- \frac{(2-p)c(\L_{N,p,s})|\O|^{\frac{p}{p^*_s}-\frac{p}{2}}T}{\parallel u_0\parallel^{2-p}_2}\Big)^{\frac{1}{2-p}}
$$
Hence if $T<T^*$, $u(x,T)=0$ and the result follows.

Assume that $1<p<\frac{2N}{N+2s}$, using an approximation argument, we can take $u^{\nu}$ as test function in \eqref{propa}, it holds that
\begin{eqnarray*}
&\dyle \frac{1}{\nu+1}\frac{d}{dt}\int_{\Omega }u^{\nu+1}dx+\frac{1}{2}\iint_{D_{\O}}|u(x,t)-u(y,t)|^{p-2}(u(x,t)-u(y,t))(u^\nu(x,t)-u^\nu(y,t))d\nu\\
& = \dyle\l\int_{\Omega }\dfrac{u^{p-1+\nu}}{|x|^{ps}}dx.
\end{eqnarray*}
Hence, by inequality\eqref{alge3}, we get
\begin{equation*}
\frac{1}{\nu+1}\frac{d}{dt}\int_{\Omega }u^{\nu+1}dx+\frac{C}{2}\iint_{D_{\O}}|u^{\frac{p+\nu-1}{p}}(x,t)-u^{\frac{p+\nu-1}{p}}(y,t)|^{p}\,d\nu\le \l\int_{\Omega }\dfrac{u^{p-1+\nu}}{|x|^{ps}}dx.
\end{equation*}%
Using now Hardy inequality
\begin{equation*}
\frac{1}{\nu+1}\frac{d}{dt}\int_{\Omega }u^{\nu+1}dx+(\frac{C}{2}-\frac{\l}{\L_{N,p,s}})\iint_{D_{\O}}|u^{\frac{p+\nu-1}{p}}(x,t)-u^{\frac{p+\nu-1}{p}}(y,t)|^{p}\, d\nu\le 0.
\end{equation*}%
Assume that $\l<\dfrac{C\L_{N,p,s}}{2}$, hence by using Sobolev inequality, we conclude that
\begin{equation*}
\frac{1}{\nu+1}\frac{d}{dt}\int_{\Omega }u^{\nu+1}dx+C(\L_{N,p,s})\left( \int_{\Omega }u^{\frac{\left( \nu+p-1\right) }{p}p^{\ast }_s}dx\right) ^{\frac{p}{p^{\ast }_s}}\leq 0.
\end{equation*}
Recall that $\nu=\frac{N(2-p)-ps}{ps}$, then  $\frac{\nu+p-1}{p}p^{\ast }_s=\nu+1 $.\\

$$\frac{1}{\nu+1}\frac{d}{dt} ||u(x,t)||^{\nu+1}_{\nu+1}+C||u(x,t)||^{\nu+p-1}_{\nu+1} \le 0$$
Now,  we get that
$$
||u(x,T)||_{\nu+1}\le ||u_0||_{\nu+1}\Big(1-\frac{C T}{\parallel u_0\parallel^{2-p}_{\nu+1}}\Big)^{\frac{1}{2-p}}.
$$
Hence the result follows.
\end{proof}

Now, for the more general problem
\begin{equation}\label{concave}
\left\{
\begin{array}{rcll}
u_t+(-\D^s_{p}) u & = & \dyle\l\frac{u^{p-1}}{|x|^{ps}}+u^q  &
\text{ in } \O_{T}=\Omega \times (0,T)  , \\
u & \ge & 0 &
\text{ in }\ren\times (0,T), \\
u &= & 0 & \text{ in }(\ren\setminus\O) \times (0,T), \\
u(x,0) & = & u_0(x) & \mbox{  in  }\O,
\end{array}%
\right.
\end{equation}
where $q\le 1$, as in Theorem \ref{exis2}, we can prove that \eqref{concave} has a solution with finite time extension, more precisely we have
\begin{Theorem}\label{concave-extinction}
Assume that $p-1<q\le 1, \l <\L_{N,p,s}$ and $u_0\in L^2(\O)$. Then problem \eqref{concave} has a nonnegative
 minimal solution $u\in L^p(0,T; W^{s,p}_{ 0}(\Omega))$, moreover
 \begin{enumerate}
 \item if $p\ge \frac{2N}{N+2s}$,  then under a smallness condition
on $||u_0||_2$, there exists a finite time  $T^*$ such that $u(.,t)\equiv 0$ for all $t\ge T^*$.
\item if $1<p<\frac{2N}{N+2s}$, $p-1<q\leq 1$ and $u_0 \in L^{\nu+1}(\O)\cap  L^{2}(\O)$ with $\nu+1=\frac{N(2-p)}{ps}$, then there exists $C>0$ such that if $\l<C$, then $u( .,t) \equiv 0$ for all $t\geq T^*$ for some $T^*>0$.
\end{enumerate}
\end{Theorem}

\

If $q<p-1$, then as in the local case, a different phenomenon appears and non extinction in finite time occurs. More precisely we have the following result.

\begin{Theorem}\label{nonexis2} Assume that $1<p<2, \l\le \L_{N,p,s}$ and let $q<p-1$, then the problem
\begin{equation}\label{non1}
\left\{\begin{array}{rcll} u_t-\D^s_pu &=& \l\dfrac{u^{p-1}}{|x|^{ps}}+ u^{q} & \mbox{ in } \O\times (0,T),\\
 u&=&0 &\hbox{  in
\  } (\ren\backslash\O)\times (0,T),\\
u(x,0)&=& 0 & \mbox{ in }\O,

\end{array}
\right.
\end{equation}
has a global solution $u$ such that $u(x,t)>0$ for all $t>0$ and
$x\in \O$, namely there is non finite time extinction, moreover, $u(.,t)\uparrow w$ as $t \to \infty$ where
$w$ is the unique positive solution to problem
\begin{equation}\label{elnon1}
\left\{\begin{array}{rcll}-\D^s
_pw&=& \l\dfrac{w^{p-1}}{|x|^{ps}}+w^q & \mbox{ in } \O,  \\w&=&0 & \hbox{  in
\  } \ren\backslash\O.
\end{array}
\right.
\end{equation}
\end{Theorem}
\section{The case $p>2$ and $\l>\L_{N,p,s}$: Non existence result}\label{sec:lineal01}
In \cite{GP}, for the local case, the authors proved that if $p>2$, $\l>\L_{N,p,1}$ and $u_0\ge C$ in some ball $B_\eta(0)$, then problem \eqref{propa}  has non negative solution in the sense that if we consider $u_n$ to be the unique solution to problem  \eqref{pro:lineal1}, then for all $\e>0$, there exists $r(\e)>0$ such that $u_n(x,t)\to \infty$ as $n\to\infty$ if $|x|<r(\e)$ and $t>\e$. This phenomenon occur since the parabolic operator has the finite speed propagation and then if $\text{Supp}(u_0)\subset \O\backslash B_\eta(0)$, then for $t$ small, the Hardy potential has non effect and then the solution can exists for small $t$.

Since in our case, the nonlocal operator has not the finite speed propagation, we will show that problem \eqref{propa}  has non solution in a suitable sense.

Let us begin by the next property of the Hardy constant defined in \eqref{LL}. If $\O$ is a bounded domain such that $0\in \O$, then we define
\begin{equation}\label{LLOO}
\L_{N,p,s, \O}=\inf\limits_{u\in W^{s,p}_0(\O)}\dfrac{\frac 12 \dyle\iint_{D_\O}
\dfrac{|u(x)-u(y)|^p}{|x-y|^{N+ps}}dxdy}{\dyle\io\dfrac{|u(x)|^p}{|x|^{ps}}dx},
\end{equation}
then from \cite{AB} we have that $\L_{N,p,s,\O}=\L_{N,p,s}$ and $\L_{N,p,s,\O}$ is not achieved.

We are now in position to state the main non existence result of this subsection.
\begin{Theorem}\label{th:non}
Let $u_0\in L^1(\O)$ be such that $u_0\gneqq 0$ and $\l >\L_{N,p,s}$,
then problem \eqref{propa} has non positive solution \textit{obtained as limit of approximations} (SOLA).
\end{Theorem}
\begin{proof}
Without loss of generality we can assume that $u_0\in L^\infty(\O)$. We argue by contradiction, suppose that problem \eqref{propa} has a solution $u\gneqq 0$ obtained as a limit of approximation. Using Monotony argument we get easily that \eqref{propa} has a nonnegative minimal SOLA solution denoted by $u$ with $u=\limit_{n\to \infty}u_n$ and $u_n$ is the unique solution to the problem
\begin{equation}\label{eq:apro11}
\left\{
\begin{array}{rcll}
u_{nt}+(-\D^s_{p}) u_n & = & \dyle \l a_n(x)u_n^{p-1} &
\text{ in } \O_{T}=\Omega \times (0,T)  , \\
u_n & \ge & 0 &
\text{ in }\ren\times (0,T), \\
u_n &= & 0 & \text{ in }(\ren\setminus\O) \times (0,T), \\
u_n(x,0) & = & u_0(x) & \mbox{  in  }\O,
\end{array}%
\right.
\end{equation}
where $a_n(x)=\min\{n,\dfrac{1}{|x|^{ps}}\}$. It is clear that $\{u_n\}_n$ is increasing in $n$ and $u_n\uparrow u$ a.e. in $\O_T$.

Since the finite speed propagation does not holds for $u_1$, see \cite{AABP}, then we get the existence of $0<t_1<t_2$ such that for all $x\in \O$ and for all $t\in [t_1,t_2]$ we have $u_1(x,t)>0$. In particular for $0<\rho<<1$ be chosen later such that $B_\rho(0)\subset\subset \O$, we have $u_1(x,t)>C>0$ for all $(x,t)\in \bar{B}_\rho(0)\times [t_1,t_2]$.

Let $\bar{\eta}<\rho$ to be chosen later, then all $n\ge 1$, $u_n(x,t)\ge c=\frac{C}{2}$ for all $(x,t)\in
B_{\eta}(0)\x (t_1, t_2)$.

Consider
$\psi \in C^{\infty}_0(B_{\bar{\eta}} (0)) $, using Theorem
\ref{pic}, we obtain that
$$
\frac 12
\dint\dint_{D_{B_{\bar{\eta}}(0)}}\dfrac{|\psi(x)-\psi(y)|^{p}}{|x-y|^{N+ps}}dx\,dy\ge \io
\frac{(-\Delta)^s_{p} u_n }{u_n^{p-1}}|\psi|^p dx.
$$
Hence
\begin{equation*} \frac 12
\iint_{D_{B_{\bar{\eta}} (0)}}\dfrac{|\psi(x)-\psi(y)|^{p}}{|x-y|^{N+ps}}dx\,dy \ge \l\int_{B_{\bar{\eta}} (0)} | \psi |
^pa_n(x)dx-\int _{B_{\bar{\eta}} (0)} |
\psi | ^p\frac{u_{nt}}{u^{p-1}_n}dx.
\end{equation*}
Integrating in time,
$$
\begin{array}{lll}
&\dyle \frac{(t_2-t_1)}{2}
\iint_{D_{B_{\bar{\eta}} (0)}}\dfrac{|\psi(x)-\psi(y)|^{p}}{|x-y|^{N+ps}}dx\,dy\\
&\ge \dyle \l(t_2-t_1)\int _{B_{\bar{\eta}}
(0)} | \psi | ^pa_n(x) dx-\frac{1}{p-2}\int_{B_{\bar{\eta}} (0)}
\frac{|\psi | ^p}{u^{p-2}_n(x,t_1) }dx.
\end{array}
$$
Thus
\begin{equation}\label{uam00}
\begin{array}{lll}
& \dyle \frac{1}{2} \iint_{D_{B_{\bar{\eta}} (0)}}\dfrac{|\psi(x)-\psi(y)|^{p}}{|x-y|^{N+ps}}dx\,dy+
\frac{1}{(p-2)(t_2-t_1)c^{p-2}}
\int_{B_{\bar{\eta}} (0)}|\psi|^pdx
\\
& \ge \dyle\l\int_{B_{\bar{\eta}} (0)}
\dfrac{|\psi|^p}{|x|^{ps}}dx.
\end{array}
\end{equation}
Now, using H\"older and Sobolev inequalities we reach that
$$
\begin{array}{lll}
\dyle\int_{B_{\bar{\eta}} (0)}|\psi|^pdx & \le &  \bigg(\dyle\int_{B_{\bar{\eta}} (0)}|\psi|^{p^*_s}dx\bigg)^{\frac{p}{p^*}}|B_{\bar{\eta}}(0)|^{\frac{p^*-p}{p^*_s}}\\
&\le &\dyle C\bar{\eta}^{ps}\iint_{D_{B_{\bar{\eta}} (0)}}\dfrac{|\psi(x)-\psi(y)|^{p}}{|x-y|^{N+ps}}dx\,dy,
\end{array}
$$
where $C$ depends only on $N,p,s$. Thus going back to \eqref{uam00} we conclude that
\begin{equation}\label{uam001}
\begin{array}{lll}
& \dyle \bigg(\frac{1}{2}+ \frac{C \bar{\eta}^{ps} }{(p-2)(t_2-t_1)c^{p-2}}\bigg) \iint_{D_{B_{\bar{\eta}} (0)}}\dfrac{|\psi(x)-\psi(y)|^{p}}{|x-y|^{N+ps}}dx\,dy \ge \dyle\l\int_{B_{\bar{\eta}} (0)}
\dfrac{|\psi|^p}{|x|^{ps}}dx.
\end{array}
\end{equation}
Since
$\l>\L_{N,p,s}$, then we can choose ${\bar{\eta}}<<\rho$ such that
\begin{equation}\label{eq:ccc} \dfrac{\l}{1+\frac{2 C \bar{\eta}^{ps} }{(p-2)(t_2-t_1)c^{p-2}}}\ge \L_{N,p,s}+\e,
\end{equation}
for some $\e>0$. Going back to \eqref{uam001} we reach a contradiction with the Hardy inequality. Hence we conclude.
\end{proof}

\begin{remark}

\

\begin{enumerate}
\item Define $v(x,t)=C(t-t_1)\log\bigg(\frac{\bar{\eta}}{|x|}\bigg)$, then $v(x,t_1)=0$ and $v$ solves
\begin{equation}\label{logg}
\left\{
\begin{array}{rcll}
v_t+(-\D^s_{p}) v & \le & \dyle\frac{C}{|x|^{ps}}  &
\text{ in } B_{\bar{\eta}}(0)\times (t_1,t_2), \\
v &\le & 0 & \text{ in }(\ren\setminus B_{\bar{\eta}}(0)) \times (t_1,t_2), \\
v(x,t_1) & = & 0 & \mbox{  in  } B_{\bar{\eta}}(0).
\end{array}%
\right.
\end{equation}
Choosing $\bar{\eta}<<1$ and $C$ small and by the comparison principle we obtain that $v\le u$ in $B_{\bar{\eta}})(0)\times (t_1,t_2)$. Thus
$$\lim_{|x|\to
0}u(x,t)=\infty\,\,\forall \,\, t\in (t_1,t_2).$$

\item Following the same argument as in the proof of Theorem \ref{th:non} we can show that the problem
\begin{equation}\label{nongeneral}
\left\{
\begin{array}{rcll}
u_t+(-\D^s_{p}) u & = & \dyle\l\frac{u^{\a}}{|x|^{ps}}  &
\text{ in } \O_{T}=\Omega \times (0,T)  , \\
u & \ge & 0 &
\text{ in }\ren\times (0,T), \\
u &= & 0 & \text{ in }(\ren\setminus\O) \times (0,T), \\
u(x,0) & = & u_0(x) & \mbox{  in  }\O,
\end{array}%
\right.
\end{equation}
with $p>2$, $\lambda>0$  and $\a>p-1$, has nonnegative solution \textit{obtained as limit of approximations} (SOLA).
\end{enumerate}
\end{remark}
\section{Appendix}
We give here a detailed proof of the algebraic Lemma \ref{real11}.

{\bf Proof of Lemma \ref{real11}.}
If $(b_1,b_2)=(0,0)$ or $(a_1,a_2)=(0,0)$, then \eqref{alge4} holds trivially.

If $a_1=a_2$, then \eqref{alge4} holds for any $b_1, b_2\ge 0$ choosing $C_1\le C_2$.

Assume that $(b_1,b_2)\neq(0,0)$, $(a_1,a_2)\neq(0,0)$ and $a_1\neq a_2$.

We divide the proof in several cases.

{\bf I-The first case: $a_1>a_2\ge 0$.} We set $\delta=\dfrac{a_2}{a_1}\in [0,1)$, then in this case, \eqref{alge4} is equivalent to
\begin{equation}\label{alge411}
(1-\d)^{p-1}(b_1- \d b_2)\ge C_1 |b^{\frac{1}{p}}_1-\d b^{\frac{1}{p}}_2|^p-C_2 |b^{\frac{1}{p}}_1-b^{\frac{1}{p}}_2|^p.
\end{equation}
If $b_1=b_2$, then \eqref{alge411} take the form
$$
(1-\d)^{p}\ge C_1 (1-\d)^p
$$
that holds trivially since $C_1\le 1$. Thus we assume that $b_1\neq b_2$.

\begin{itemize}
\item  {\emph{Sub-case 1: $b_1>b_2\ge 0$.}}
We set $\theta=(\frac{b_2}{b_1})^{\frac{1}{p}}\in [0,1)$, then \eqref{alge411} take the form
\begin{equation}\label{alge4110}
(1-\d)^{p-1}(1- \d \theta^p)\ge C_1 (1-\d \theta)^p-C_2 (1-\theta)^p.
\end{equation}
We have
$$
1-\d \theta=(1-\theta)+\theta(1-\d).
$$
Thus
\begin{eqnarray*}
(1-\d \theta)^p & = & ((1-\theta)+\theta(1-\d))^p\\
&\le & (1+\e)^{p-1}\theta^p(1-\d)^p+(1+\frac{1}{\e})^{p-1}(1-\theta)^p
\end{eqnarray*}
where $\e>0$ is any positive constant.
Now, using the fact that $(1-\d)\theta^p<1-\d\theta^p$, we obtain that $\theta^p(1-\d)^p\le (1-\d)^{p-1}(1-\d\theta^p)$. Therefore using Young inequality we conclude that
$$
C_1(1-\d \theta)^p\le C_1(1+\e)^{p-1}(1-\d)^{p-1}(1-\d\theta^p)+C_1(1+\frac{1}{\e})^{p-1}(1-\theta)^p.
$$
It is clear that we can choose $C_1<1$ depending only on $\e$ such that $C_1(1+\e)^{p-1}=1$, thus
$$
C_1(1-\d \theta)^p\le (1-\d)^{p-1}(1-\d\theta^p)+C_2(1-\theta)^p
$$
where $C_2=\max\{1, C_1(1+\frac{1}{\e})^{p-1}\}$ and then \eqref{alge4110} holds in this case.

\item {\emph{Sub-case 2: $b_2>b_1\ge 0$.}}
In this case we set $\theta=(\frac{b_1}{b_2})^{\frac{1}{p}}\in [0,1)$, thus
\eqref{alge411} take the form
\begin{equation}\label{alge41101}
(1-\d)^{p-1}(\theta^p-\d)\ge C_1 |\theta-\d|^p-C_2 (1-\theta)^p.
\end{equation}
We divide the proof of \eqref{alge41101} into two cases:
\begin{enumerate}
\item {\bf Sub-sub-case i: $\theta^p>\d$.}
It is clear that $\d<\theta^p<\theta$, then we have
\begin{eqnarray*}
(\theta-\d)^p & = & ((\theta-\theta^p)+(\theta^p-\d))^p\\
&\le & (1+\e)^{p-1}(\theta^p-\d)^p+(1+\frac{1}{\e})^{p-1}(\theta-\theta^p)^p\\
&\le & (1+\e)^{p-1}(\theta^p-\d)^{p-1}(\theta^p-\d) +(1+\frac{1}{\e})^{p-
1}\theta^p(1-\theta^{p-1})^p.
\end{eqnarray*}
Since $1-\theta^{p-1}\le 1-\theta$ and $\theta^p-\d\le 1-\d$, then using the same hypothesis on $C_1, \e$ and $C_2$ as in the previous case, it follows that
$$
C_1(\theta-\d)^p \le (1-\d)^{p-1}(\theta^p-\d) +C_2(1-\theta)^p
$$
and then \eqref{alge41101} follows.
\item {\bf Sub-sub-case ii: $\theta^p\le \d$.} This is the more delicate case and it need some fine computations. It is clear that in this case we have to show that
    \begin{equation}\label{alge411010}
C_2 (1-\theta)^p\ge C_1 |\theta-\d|^p+(1-\d)^{p-1}(\d-\theta^p).
\end{equation}
Let begin by assuming that {\textit{$\theta<\d$}}, then trivially we have $C_1 |\theta-\d|^p\le (1-\theta)^p$. Hence we have just to show that
$$
(1-\d)^{p-1}(\d-\theta^p)\le C_2 (1-\theta)^p.
$$
For $\rho\in (0, \d)$, we define the function $h$ by
$$
h(\rho)=(1-\rho)^p+(1-\d)^{p-1}\rho^p.
$$
It is clear that $h'(\rho)=p\bigg(\rho^{p-1}(1-\d)^{p-1}-(1-\rho)^{p-1}\bigg)\le 0$. Since $\theta\le \d$, we conclude that $h(\theta)\ge h(\d)$ and then we conclude that
$$
(1-\theta)^p+(1-\d)^{p-1}\theta^p\ge (1-\d)^p+(1-\d)^{p-1}\d^p.
$$
Thus
$$
(1-\theta)^p\ge (1-\d)^{p-1}(1-\d-\theta^p +\d^p).
$$
Since $p<2$, then it is not difficult to show that $1-\d+\d^p\ge \d$, therefore we reach that
$$
(1-\theta)^p\ge (1-\d)^{p-1}(\d-\theta^p)
$$
and the result follows.

Assume now that $\d\le \theta\le\d^{\frac{1}{p}}$. As above for
$\rho\in (0, \theta)$, we define the function $h_1$ by
$$
h_1(\rho)=(1-\rho)^{p-1}(\rho-\theta^p),
$$
then
\begin{eqnarray*}
h'(\rho) &= & (1-\rho)^{p-2}\bigg(1-p\rho+(p-1)\theta^p\bigg)\\
&\ge & (1-\rho)^{p-2}\bigg(1-p\rho+(p-1)\rho^p\bigg).
\end{eqnarray*}
Since $p<2$, we have $(1-p\rho+(p-1)\rho^p)\ge 0$, thus $h'_1\ge 0$ and then $h_1(\d)\le h_1(\theta)$.
Hence
$$
(1-\d)^{p-1}(\d-\theta^p)\le (1-\theta)^{p-1}(\theta-\theta^p)\le (1-\theta)^p,
$$
where the last inequality follows using the fact that $2\theta\le 1+\theta^p$.
\end{enumerate}
\end{itemize}
\

{\bf II-The second case: $a_2>a_1\ge 0$.}
It is clear that
$$
|a_1-a_2|^{p-2}(a_1-a_2)(a_1b_1-a_2b_2)=|a_2-a_1|^{p-2}(a_2-a_1)(a_2b_2-a_1b_1),
$$
thus using the result of the first case, it follows that
\begin{eqnarray*}
|a_1-a_2|^{p-2}(a_1-a_2)(a_1b_1-a_2b_2) & \ge &  C_1 |a_2b^{\frac{1}{p}}_2-a_1b^{\frac{1}{p}}_1|^p-C_2(\max\{|a_2|, |a_1|\})^p|b^{\frac{1}{p}}_2-b^{\frac{1}{p}}_1|^p\\
&\ge & C_1 |a_1b^{\frac{1}{p}}_1-a_2b^{\frac{1}{p}}_2|^p-C_2(\max\{|a_1|, |a_2|\})^p|b^{\frac{1}{p}}_1-b^{\frac{1}{p}}_2|^p.
\end{eqnarray*}
and then we conclude.

\

{\bf III-The third case: $a_1, a_2\le 0$.} We set $\tilde{a}_1=-a_1$ and $\tilde{a}_2=-a_2$, then $\tilde{a}_1, \tilde{a}_2\ge 0$ and
$$
|a_1-a_2|^{p-2}(a_1-a_2)(a_1b_1-a_2b_2)=|\tilde{a}_1-\tilde{a}_2|^{p-2}(\tilde{a}_1-\tilde{a}_2)(\tilde{a}_1b_1-\tilde{a}_2b_2),
$$
then the result follows using the previous cases.

\

{\bf IV-The fourth case: $a_1<0<a_2$ or $a_2<0<a_1$.} Let assume that $a_2<0<a_1$ and define $\tilde{a}_2=-a_2$, we get
$$
|a_1-a_2|^{p-2}(a_1-a_2)(a_1b_1-a_2b_2)=(a_1+\tilde{a}_2)^{p-1}(a_1b_1+\tilde{a}_2 b_2).
$$
As in the previous case, without loss of generality we can assume that $a_1\ge \tilde{a}_2$ and $b_1\ge b_2$. Setting
$\delta=\frac{\tilde{a}_2}{a_1}, \theta=(\frac{b_2}{b_1})^{\frac{1}{p}}\in [0,1)$ then \eqref{alge4} is equivalent to
\begin{equation}\label{alge511}
(1+\d)^{p-1}(1+\d \theta^p)\ge C_1 (1+\d \theta)^p-C_2 (1-\theta)^p.
\end{equation}
We have
\begin{eqnarray*}
C_1(1+\d \theta)^p &\le & C_1(1+\d \theta^p-\d\theta^p+\d\theta)^p\\
&\le & C_1(1+\e)^{p-1}(1+\d\theta^p)^{p}+C_1(1+\frac{1}{\e})^{p-1}\d^p\theta^p(1-\theta^{p-1})^p\\
&\le & (1+\d)^{p-1}(1+\d\theta^p)+C_2(1-\theta)^p,
\end{eqnarray*}
where, as in the previous cases, we have used the fact that $C_1(1+\e)^{p-1}=1\le C_2$. Hence the result follows. \cqd

\end{document}